\documentclass[a4paper]{article}
\usepackage[all]{xy}\usepackage[latin1]{inputenc}        
\usepackage[dvips]{graphics,graphicx}
\usepackage{amsfonts,amssymb,amsmath,color,mathrsfs, amstext}
\usepackage{amsbsy, amsopn, amscd, amsxtra, amsthm,authblk}
\usepackage{enumerate,algorithm,algpseudocode}
\usepackage{fancyhdr}
\usepackage{setspace}
\usepackage{ulem}
\usepackage{bbm}

\usepackage{upref,ulem}
\usepackage{geometry}
\usepackage{float}
\usepackage{subfigure}
\usepackage{yhmath}
\usepackage{booktabs}
\usepackage{graphicx}
\usepackage{multirow}
\usepackage{makecell}
\usepackage[colorlinks,
            linkcolor=red,
            anchorcolor=red,
            citecolor=red,
            urlcolor=black,
            ]{hyperref}

\numberwithin{equation}{section}

\newcommand{\E}{\mathbb{E}}
\newcommand{\R}{\mathbb{R}}

\def\R{\mathbb{R}}

\newcommand \footnoteONLYtext[1]
{
	\let \mybackup \thefootnote
	\let \thefootnote \relax
	\footnotetext{#1}
	\let \thefootnote \mybackup
	\let \mybackup \imareallyundefinedcommand
}

\definecolor{DarkBlue}{rgb}{0,0,.8}
\newtheorem{theorem}{Theorem}[section]

\newtheorem{definition}{Definition}[section]

\newtheorem{proposition}{Proposition}[section]

\numberwithin{equation}{section}

\providecommand{\keywords}[1]
{
  \small	
  \textbf{\textit{Keywords:}} #1
}

\providecommand{\msc}[1]
{
  \small	
  \textbf{\textit{Mathematics Subject Classification:}} #1
}

\begin{document}

\title{Solving Schr\"{o}dinger bridge problem via continuous normalizing flow}

\author[1]{Yang Jing \thanks{sharkjingyang@sjtu.edu.cn}  }
\author[2]{Lei Li \thanks{leili2010@sjtu.edu.cn} }
\author[3]{Jingtong Zhang \thanks{sffred@sjtu.edu.cn}}
\affil[1,2,3]{School of Mathematical Sciences, Shanghai Jiao Tong University, Shanghai, 200240, P.R.China.}
\affil[2]{Institute of Natural Sciences, MOE-LSC, Shanghai Jiao Tong University, Shanghai, 200240, P.R.China.}
\affil[2]{Shanghai Artificial Intelligence Laboratory}

\date{}
\maketitle

\begin{abstract}
The Schr\"{o}dinger Bridge Problem (SBP), which can be understood as an entropy-regularized optimal transport, seeks to compute stochastic dynamic mappings connecting two given distributions. SBP has shown significant theoretical importance and broad practical potential, with applications spanning a wide range of interdisciplinary fields. While theoretical aspects of the SBP are well-understood, practical computational solutions for general cases have remained challenging. This work introduces a computational framework that leverages continuous normalizing flows and score matching methods to approximate the drift in the dynamic formulation of the SBP. The learned drift term can be used for building generative models, opening new possibilities for applications in probability flow-based methods. We also provide a rigorous $\Gamma-$convergence analysis for our algorithm, demonstrating that the neuron network solutions converge to the theoretical ones as the regularization parameter tends to infinity. Lastly, we validate our algorithm through numerical experiments on fundamental cases. 
\end{abstract}
\\
\msc{49Q22; 68T07}

\section{Introduction}\label{sec:intro}

The challenge of finding mappings between two probability distributions is the focus of many modern machine learning applications, from generative modeling to understanding physical systems \cite{neklyudov2023action}. While deterministic methods such as Optimal Transport (OT) have offered robust frameworks for addressing this task, many real-world phenomena are inherently governed by stochastic processes, necessitating more sophisticated approaches. The Schr\"{o}dinger Bridge Problem, originally proposed by Schr\"{o}dinger \cite{schrodinger1932theorie}, addresses this challenge by seeking the most likely stochastic process connecting two probability distributions. SBP is of great theoretical importance in mathematical physics, as they provide deeper insights into the evolution of physical systems and the study of quantum information. It also demonstrates broad practical potential across various applications in interdisciplinary fields such as modeling natural stochastic
dynamical systems \cite{schiebinger2019optimal,holdijk2024stochastic}, image processing \cite{liu20232} and shape correspondence \cite{feydy2017optimal}. 

Rigorously, the Schr\"{o}dinger problem can be stated as follow \cite{christian_leonard_survey_2014}: given two distribution $\rho_0$ and $\rho_1$ on a bounded domain $\Omega$, we aim to find a path measure $\mathbb{P}^{\ast}$ such that 
\begin{equation}\label{eq:dynschb}
    \mathbb{P}^{\ast}=\mathop{\mathrm{argmin}}\limits_{\mathbb{P} \in \mathcal{D}(\rho_0,\rho_1)} \mathbb{KL}(\mathbb{P}\|\mathbb{Q}),
\end{equation}
where $\mathbb{Q}$ is a reference measure, $\mathbb{P}$ belongs to a set of path measures $ \mathcal{D}(\rho_0,\rho_1) \subset \mathcal{P}(C([0,T]; \Omega))$, with marginal measures $\rho_0$ at time $t=0$ and $\rho_1$ at time $t=T$. Here, $\mathcal{P}(C([0, T];\Omega))$ indicates the set of measures on the path space $C([0, T];\Omega)$ (continuous curves taking values in $\Omega$). The KL divergence between two probability measure $\mu$ and $\nu$ on $\Omega$ is defined by
\begin{equation*}
    \mathbb{KL}\left[\mu \| \nu\right]=\left\{\begin{array}{lr}
    \int_{\Omega} \log (\frac{d \mu}{d \nu}) d\mu, & \text { if } \mu \ll \nu ,\\
    \ \infty, & \text { else, }
    \end{array}\right.
\end{equation*}
where $\frac{d\mu}{d \nu}$ denotes the Radon-Nikodym derivative of $\mu$ with respect to $\nu$. Note that the KL divergence is non-negative by Jensen's inequality and achieves zero only if $\mu= \nu$. Moreover, it is a convex functional with respect to both arguments. In practice, $\mathbb{Q}$ is often taken to be the path measure induced by Brownian motion or some relatively simple diffusion process.

Previous research has demonstrated a close relationship between SBP and the optimal stochastic control problem \cite{dai1991stochastic}, thereby offering an alternative approach to solving the SBP. Given that $\mathbb{Q}=\sqrt{2} \sigma \mathbb{W}$, where $\mathbb{W}$ is the path measure induced by the standard Brown motion,then solving SBP \eqref{eq:dynschb} is equivalent to finding the optimal stochastic control of the following optimization problem \cite{dai1991stochastic}:
\begin{equation}\label{eq:SBPstandardQ}
\begin{array}{ll}
& \min\limits_{u} \mathbb{E}\left[\frac{1}{2} \int_0^T \|u_t(X_t,t)\|^2 dt\right], \\
\text { s.t. } & dX_t = u_t(X_t,t) dt + \sqrt{2}\sigma\,dW, \\
& X_0\sim\rho_0, X_T\sim\rho_1.\\
\end{array}
\end{equation}
Or alternatively in a form with Fokker-Planck equaiton constraints:
\begin{equation}\label{eq:schb}
\begin{array}{ll}
\min\limits_{\rho, u} & \frac{1}{2}\int_{0}^{T} \int_{\Omega}\rho(x,t) |u(x,t)|^2 dx dt, \\
\text { s.t. } &\partial_{t} \rho +\nabla \cdot(\rho u)= \sigma^2 \Delta \rho ,\\
&\rho(x,0)=\rho_{0}(x) ,\ \rho(x,T)=\rho_{1}(x) \\
\end{array}
\end{equation}
Under mild assumptions on $\rho_0$ and $\rho_1$, Schr\"{o}dinger bridge probelm \eqref{eq:schb} admits a unique minimizer $(\rho,v=\nabla \Phi)$ \cite{christian_leonard_survey_2014}, where $(\rho,\Phi)$ is the classical solutions of the Euler-Lagrange equations:
\begin{equation}
    \left\{\begin{aligned} 
    &\partial_t \rho+\nabla \cdot(\nabla \Phi \rho)=\sigma^2 \Delta \rho, \\
    & \partial_t \Phi +\frac{1}{2}|\nabla \Phi|^2=-\sigma^2 \Delta \Phi .
    \end{aligned}
    \right.
\end{equation}

If $\mathbb{Q}$ is the measure generated by the Langevin diffusion $dX_t= -\nabla U(X_t) dt + \sqrt{2}\sigma dW$, where $U(x)$ is a given potential, then solving corresponding SBP  is equivalent to finding the optimal control $u$ such that the diffusion process $dX_t= \left[u(X_t,t)-\nabla U(X_t) \right]dt + \sqrt{2}\sigma dW$ bridges the distributions $\rho_0$ and $\rho_1$ and minimizes the $L^2$ cost:
\begin{equation}\label{problem:SOC}
\begin{array}{ll}
    \min\limits_{\rho, u} & \frac{1}{2}\int_{0}^{T} \int_{\Omega}\rho(x,t) |u(x,t)|^2 dx dt, \\
\text { s.t. } &\partial_{t} \rho +\nabla \cdot(\rho (u-\nabla U))= \sigma^2 \Delta \rho ,\\
&\rho(x,0)=\rho_{0}(x) ,\ \rho(x,1)=\rho_{1}(x) \\
\end{array}
\end{equation}

Traditional approaches to solving the SBP have primarily relied on Iterative Proportional Fitting \cite{kullback1968probability} and Sinkhorn-Knopp algorithms \cite{sinkhorn1967concerning,vargas2021solving}. Iterative methods suffer from cumulative numerical errors \cite{fernandes2022shooting} and, crucially, the reliance on mesh-based discretization makes them computationally intractable for high-dimensional problems. This computational barrier has restricted the practical applications of SBP, particularly in modern machine learning contexts where high-dimensional data is prevalent. With the advancements in deep learning technologies, leveraging neural networks to efficiently solve the SBP not only accelerates the computational process but also provides novel solutions for complex systems that are difficult for traditional methods to handle, thus driving research and technological progress in related fields.

Recent advances in deep learning, particularly in the realm of continuous normalizing flows (CNFs) and diffusion models, have opened new possibilities for addressing such challenges. These approaches have demonstrated remarkable success in modeling complex probability distributions and have natural connections to differential equations \cite{song2020score,chen2018neural}. Continuous normalizing flows can empirically model a distribution measure, offering a means to approximate density evolution in scenarios such as the Fokker-Planck equation \cite{liu2023pinf} and the Benamou-Brenier formulation of optimal transport \cite{finlay2020learning}. One can model the distribution using CNFs and enforce the constraints of the associated PDEs (such as Fokker-Plank equation or continuity equation) by minimizing the residuals \cite{zeng2023adaptive,feng2021solving}, or incorporate regularization during CNF training to ensure the minimizers converge to the target solutions \cite{finlay2020train,onken2020ot}. Thus the framework of CNFs is particularly promising for the SBP due to its ability to track density evolution and its natural alignment with transport-based formulations.

In this work, we aim to solve the SBP using the framework of CNFs, which models the evolution of particles in an 'optimal' manner. We will employ the dynamic formulation of the SBP and use Monte-Carlo method to approximate the loss function. The SDE backbone of the SBP dynamic formulation introduces challenges in estimating density using the change of variables formula (see the discussion in Sec \ref{sec:Derivation}). To address this, we propose a hypothetical velocity field to transform the Fokker-Planck equation into the transport equation form, aligning the SBP with the framework of CNFs. Building on this, our algorithm introduces a novel mesh-free framework that leverages CNFs to solve the dynamic formulation of the SBP. This approach offers several key advantages:
\begin{itemize}
\item It provides a theoretically grounded method for finding the connecting stochastic differential equation (SDE) without requiring mesh discretization, making it scalable to higher dimensions.

\item Unlike traditional numerical methods, our approach learns a continuous representation of the solution, enabling efficient sampling and density estimation at arbitrary points in the probability path.

\item The framework naturally incorporates the stochastic aspects of the SBP through a combination of CNFs and score matching, providing a more faithful representation of the underlying probabilistic dynamics.
\end{itemize}

A portion of our algorithm draws inspiration from a series of works on regularized CNFs. By recognizing that both optimal transport and CNFs can be framed within a transport framework, Finlay et al.\cite{finlay2020train} were the first to introduce optimal transport regularization in CNFs and Onken et. al. \cite{onken2021ot} proposed OT-Flow, an improved version of the CNFs combined with optimal transport. The added regularization enforces the solutions to approximate the theoretical solutions of the traditional OT problem. Jing et al. \cite{jing2024machine} also employed a similar methodology to design a machine learning framework for learning geodesics under the spherical Wasserstein-Fisher-Rao metric based on unbalanced optimal transport theory. Note that the SBP can be interpreted as an entropy-regularized optimal transport problem, which motivates us to approach SBP using a framework similar to that of regularized CNFs.

Another fundamental challenge in adapting machine learning approaches to classical mathematical problems lies in ensuring that the computational solutions faithfully approximate the theoretical ones. In our context, while CNFs provide a feasible computational framework, we must establish that our numerical solutions converge to the true solution of the SBP. We address this through the mathematical framework of $\Gamma$-convergence, which allows us to study the convergence of minimizers in optimization problems. By formulating our computational approach as a sequence of relaxed optimization problems with increasing penalty on the terminal distribution constraint, we can rigorously demonstrate that the solutions of our algorithm converge to those of the classical SBP. This theoretical foundation bridges the gap between computational efficiency and mathematical rigor.

The rest of the paper is organized as follows. Section \ref{sec:preliminaries} is devoted to a brief review of the necessary concepts and tools established in previous works. Next, we present our approach and derive the necessary equations to solve the dynamic formulation of the SBP in Section \ref{sec:Derivation}. In Section \ref{sec:algorithm}, we propose our algorithm to solve the SBP, including training hypothetical velocity field via CNFs and recovering optimal control through score matching. A convergence analysis is developed in Section \ref{sec:Convergence}. In section \ref{sec:numericalexp}, we provide some numerical experiments to validate the algorithm.

\section{Preliminaries}\label{sec:preliminaries}

In this section, we review several necessary concepts, tools and related preliminary results for our use later. In particular, we will introduce the  CNFs, the $\Gamma$-convergence and the Benamou-Brenier functional.

\subsection{Continuous normalizing flows}

As mentioned above, we will develop a practical algorithm for the SBP using CNFs.  
The CNFs \cite{chen2018neural,finlay2020train,grathwohl2018ffjord} are a class of generative models based purely on particle transportation. They aim to construct continuous and invertible mappings between an arbitrary distribution $\rho_0$ and the standard reference distribution $\rho_1$ through the flows induced by a velocity field $v(\cdot, t)$, which is represented by a neural network.

For a given time $T$, one is trying to obtain a mapping $z: \mathbb{R}^{d} \times [0,T]\rightarrow \mathbb{R}^{d}$, which defines a continuous evolution $x \mapsto z(x, t)$ driven by the velocity field for every $x \in \R^{d}$ such that the distribution at $T$ matches $\rho_1$.  Define $\ell(x, t):=\log |\det \nabla z(x, t)|$, then $z(x,t)$ and $\ell(x,t)$ satisfy the following ODE system
\begin{equation}\label{eq:CNFODE}
\partial_{t}\left[\begin{array}{c}
z(x, t) \\
\ell(x, t)
\end{array}\right]=\left[\begin{array}{c}
v(z(x, t), t ; \boldsymbol{\theta}) \\
\operatorname{tr}(\nabla v(z(x, t), t ; \boldsymbol{\theta}))
\end{array}\right], \quad\left[\begin{array}{c}
z(x, 0) \\
\ell(x, 0)
\end{array}\right]=\left[\begin{array}{c}
x \\
0
\end{array}\right].
\end{equation}
Then the density $\rho(\cdot,t)=(z(\cdot, t))_{\#}\rho_0$ (the pushforward of $\rho_0$ under the flow map $z(\cdot, t)$) satisfies
\begin{equation}\label{originaldensityformulation}
    \log \rho_{0}(x)=\log \rho(z(x,t),t)+\log|\det \nabla z(x,t)| \quad \text{for all} \quad x \in \R^{d}.
\end{equation}
To train the dynamics, CNFs  minimize the expected negative log-likelihood given by the right-hand-side in \eqref{eq:CNFODE}, or equivalently the KL divergence between target distribution and final distribution under the constraint \eqref{eq:CNFODE}\cite{rezende2015variational,papamakarios2017masked,papamakarios2021normalizing,grathwohl2018ffjord}:
\[
J=\mathbb{KL}\left[\rho(\boldsymbol{z}(\boldsymbol{x}, T)) \| \rho_1(\boldsymbol{z}(\boldsymbol{x}, T))\right].
\]
Notably, CNFs exhibit strong capability in modeling density changes, and we leverage this advantage of CNFs to develop our algorithm for solving the SBP.

\subsection{\texorpdfstring{$\Gamma-$} \ convergence}

Another main part of this work is to establish the theoretical foundation for our approach. Especially, we employ $\Gamma$-convergence theory to demonstrate that the relaxed formulation converges to the original SBP. The $\Gamma$-convergence has been proved to be a powerful framework to analyze the convergence of optimization problems and their corresponding optimizers. 
\begin{definition}\label{def:gamma}
Let $X$ be a topological space. Let $(f_n)$ be a sequence of functionals on $X$. Define
\begin{equation}
\begin{aligned}
    \Gamma-\limsup_{n\to\infty} f_n(x)=\sup_{N_x}\limsup_{n\to\infty}\inf_{y\in N_x}f_n(y),\\
    \Gamma-\liminf_{n\to\infty} f_n(x)=\sup_{N_x}\liminf_{n\to\infty}\inf_{y\in N_x}f_n(y),\\
\end{aligned}
\end{equation}
where $N_x$ ranges over all neighbourhoods of $x$. If there exists a functional $f$ defined on $X$ such that
\begin{equation*}
    \Gamma-\limsup_{n\to\infty}f_n = \Gamma-\liminf_{n\to\infty}f_n=f,
\end{equation*}
then we say the sequence $(f_n)$ $\Gamma$-converges to f.
\end{definition}

\begin{proposition}
Any cluster point of the minimizers of a $\Gamma$-convergent sequence $(f_n)$ is a minimizer of the corresponding $\Gamma$-limit functional $f$.
\end{proposition}

The properties of $\Gamma$-convergence make it an optimal framework for analyzing the convergence of minimizers across a sequence of optimization problems. While direct verification of $\Gamma$-convergence through Definition \ref{def:gamma} often presents significant technical challenges, the following two propositions offers an alternative approach to establishing $\Gamma$-convergence. 

\begin{proposition}\label{prop:connection}
Suppose that $X$ is a first-countable topological space. It holds that
\begin{gather*}
\begin{aligned}
&\inf_{x^n \rightarrow x} \limsup_{n\rightarrow \infty} f_n(x_n) =\Gamma\hbox{-}\limsup_{n\to\infty} f_n ,\\
&\inf_{x^n \rightarrow x} \liminf_{n\rightarrow \infty} f_n(x_n) =\Gamma\hbox{-}\liminf_{n\to\infty} f_n ,
\end{aligned}
\end{gather*}
where $\inf_{x^n\to x}$ indicates that the infimum is taken over all sequences $\{ x^n \}$ that converge to $x$. Consequently, if 
\begin{equation*}
\inf_{x^n \rightarrow x} \limsup_{n\rightarrow \infty} f_n(x_n)=\inf_{x^n \rightarrow x} \liminf_{n\rightarrow \infty} f_n(x_n) := f
\end{equation*} exists, then $(f_n)$ is $\Gamma$-convergent to $f$.
\end{proposition}

The following gives a criterion for the $\Gamma$-convergence of a special case which is applicable for our problem later.
\begin{proposition}\label{pro:gammaconv}
Let $f_n$ and $f$ be functionals defined on a first-countable topological space $X$, s.t. $f_n \uparrow f$ pointwise, and $f_n$ is lower-semicontinuous. Then $f_n$ is $\Gamma$-convergent to $f$.
\end{proposition}
The proof of Proposition \ref{pro:gammaconv} is attached in Appendix \ref{app:severalproofs}.

\subsection{Benamou-Brenier functional}
The Benamou-Brenier functional gives the rigorous definition of the $L^p$ norm $\int |v|^p\rho(dx)$ of the velocity field 
with general probability measure $\rho$ in the optimal transport related problems, which is defined as follows:
\begin{equation}
    \mathscr{B}_p(\rho, m):= \sup \left\{ \int_{X} a d\rho + \int_{X} b \cdot dm  : (a,b) \in C_b(X;K_q) \right\},
\end{equation}
$K_q:=\left\{(a, b) \in \mathbb{R} \times \mathbb{R}^d: a+\frac{1}{q}|b|^q \leq 0\right\}, \frac{1}{p}+\frac{1}{q}=1$.

The following proposition characterizes key properties of the Benamou-Brenier functional.
\begin{proposition}\label{prop:benamou}
The Benamou-Brenier functional  $\mathscr{B}_p$ is convex and lower semi-continuous on the space $\mathscr{M}(X) \times \mathscr{M}^{d}(X)$ with the weak convergence topology. Moreover, the following properties hold:
\begin{itemize}
\item  $\mathscr{B}_p(\rho,m) \geq 0$
\item if both $\rho$ and $m$ are absolutely continuous with respect to a same positive measure $\lambda$ on $X$, we can write $\mathscr{B}_p(\rho,m)=\int_X f_{p}(\rho(x),m(x))d \lambda(x)$, where we identify $\rho(x)$ and $m(x)$ are the densities with respect to $\lambda$, and $f_p:\mathbb{R} \times \mathbb{R}^d \to \mathbb{R} \cup \{\infty\}$ is defined as:
\begin{equation*}
f_p(t, x):=\sup _{(a, b) \in K_q}(a t+b \cdot x)= \begin{cases}\frac{1}{p} \frac{|x|^p}{t^{p-1}} & \text { if } t>0, \\ 0 & \text { if } t=0, x=0, \\ +\infty & \text { if } t=0, x \neq 0, \text { or } t<0 .\end{cases}
\end{equation*}
\item $\mathscr{B}_p(\rho,m) < +\infty$ only if $\rho \geq 0$ and $m \ll \rho $ 
\item for $\rho \geq 0$ and $m \ll \rho$, we have $m=u \cdot \rho$ and $\mathscr{B}_p(\rho,m)=\int \frac{1}{p} |u|^{p}d \rho$
\end{itemize}
\end{proposition}
In Section \ref{sec:Convergence}, we will employ the Benamou-Brenier functional to define the relevant functionals and examine the properties of feasible points.

\section{Governing equations in the CNF framework}\label{sec:Derivation}

The dynamic formulation \eqref{eq:SBPstandardQ} (or \eqref{problem:SOC}) of the Schr\"{o}dinger Bridge problem naturally leads to a stochastic differential equation framework similar to the CNF setting. Below, we will mainly focus on \eqref{eq:SBPstandardQ} as the example, as the generalization to others is similar.

We begin by considering the following SDE:
\begin{equation}\label{SchB_SDE}
    d X_t= u(X_t,t) dt + \sqrt{2} \sigma d W_t.
\end{equation}
Similar to \eqref{originaldensityformulation}, the law at $t$ is given by
\begin{gather}
\rho(z, t)=\E_{W}[\rho_0(X_t^{-1}(z)) (\det\nabla X_t(x))^{-1}],
\end{gather}
where $\E_W$ indicates the expectation over the stochastic mapping $x\mapsto z=X_t(x)$, given by the strong solution of the SDE \eqref{SchB_SDE}. Exactly due to this expectation, different stochastic trajectories are then coupled together, making the loss construction using the KL divergence in \cite{chen2018neural,onken2020ot} infeasible. One may use other penalty like the maximum mean discrepancy (MMD) but the practical performance seems not excellent.
Hence, our approach is to transform the SDE into an ODE formulation that gives the same time marginal distributions. This approach allows us to leverage CNFs to develop an algorithm for solving the SBP.

Consider the corresponding Fokker-Planck equation of \eqref{SchB_SDE} for \eqref{SchB_SDE}:
\begin{equation}\label{FP_SchB}
    \partial_{t} \rho +\nabla \cdot(\rho u)= \sigma^2 \Delta \rho.
\end{equation}
By noting that $\Delta \rho = \nabla \cdot \nabla \rho$, we can rewrite \eqref{FP_SchB} in the form of continuity equation:
\begin{equation}\label{Continuituy_equation_SchB}
    \partial_{t} \rho +\nabla \cdot [ \rho (u- \sigma^2 \nabla \log \rho)]= 0.
\end{equation}

The transformation from \eqref{FP_SchB} to \eqref{Continuituy_equation_SchB} is crucial as it reveals the underlying transport structure of the problem. By expressing the diffusion term using the score function $\nabla \log \rho$, we separate the deterministic and stochastic components of the dynamics, which will be essential for our neural network approximation. The time-varying hypothetical velocity field
\begin{equation}
    f:= u-\sigma^2 \nabla \log \rho
\end{equation}
represents a crucial decomposition of the dynamics. While similar transformations have been used in related contexts \cite{ilin2024transport,shen2024entropy}, our application to the SB problem enables a novel computational approach through neural network approximation. It is important to note that $f$ is not explicit, as we do not have access to the score function $\nabla \log \rho$ at this stage. 

The complete dynamical system we derive combines three key elements: the particle positions $z(x,t)$, the log-determinant of the transformation 
\begin{equation}\label{eq:logdeter}
\ell(x,t) := \log |\det \nabla z(x,t)|
\end{equation}
and the score function
\begin{equation}\label{eq:scorealongtraj}
s(x,t):=\nabla \log \rho(z(x,t),t).
\end{equation}
This coupled system allows us to simultaneously track the evolution of particles and their associated probability density, while maintaining the necessary regularity conditions for the SB problem.

By transforming Fokker-Plank equation \eqref{FP_SchB} into the form of transport equation allows us to track the evolution of particles $z(x,t)$ with corresponding $\ell(x,t)$ along the trajectories the same as \eqref{eq:CNFODE}. One can also track the score function $s(x,t)$ along the trajectories with the help of following proposition:
\begin{proposition}\label{proposition:score_ODE}
The log-determinant of the transformation $\ell(x,t)$ introduced in \eqref{eq:logdeter} satisfies the following ODE:
    \begin{equation*}
        \partial_t \ell(x,t)= \nabla_z \cdot f(z(x,t),t).
    \end{equation*}
Meanwhile, the score function along the trajectory $s(x,t)$ defined in \eqref{eq:scorealongtraj} satisfies
    \begin{equation*}
        \partial_t s(x,t)= -\nabla (\nabla \cdot f (z(x,t), t ))- \mathcal{J}_{f}^{\top}(z(x,t),t)s(x,t),
    \end{equation*}
where $\mathcal{J}_{f}$ denotes the Jacobian matrix of the vector-valued function $ f $, and is given by $\mathcal{J}_{f}= (\nabla_z f(z, t))^{\top}$ with
\begin{equation*}
    (\nabla_z f )_{ij}=\partial_i f_j.
\end{equation*}
\end{proposition}
A proof can be found in Appendix \ref{app:severalproofs}.
Thus the whole ODE system in our framework can be given by
\begin{equation}\label{eq:original ODE}
\begin{aligned}
\partial_{t} & \left[\begin{array}{c}
z(x,t) \\
\ell (x,t)\\
s(x,t)
\end{array}\right]=\left[\begin{array}{c}
f(z, t ) \\
\nabla_z \cdot f (z(x,t), t )\\
-\nabla_z (\nabla_z \cdot f (z(x,t), t ))- \mathcal{J}_{f}^{\top}(z(x,t),t)s(x,t)
\end{array}\right], \\
&\left[\begin{array}{c}
z(x, 0) \\
\ell(x,0)\\
s(x, T)
\end{array}\right]=\left[\begin{array}{c}
x \\
0\\
\nabla \log \rho_1(z(x,T))
\end{array}\right].
\end{aligned}
\end{equation}
Here $\rho_1$ is the standard normal distribution.  The first two equations is solved from $0$ to $T$, while the third equation is supposed to be solved backward from $T$ to $0$ after having the terminal position $z(x,T)$ of particles.
Recall the dynamic formulation of Sch\"{o}dinger bridge problem, we are devoted to solve a relaxed version:
\begin{equation}\label{eq:optm}
\begin{array}{ll}
\min\limits_{\rho, v} & \int_{0}^{T} \int_{\mathbb{R}^d}\rho(x,t) |u(x,t)|^2 dx dt +\alpha \mathbb{KL}[\rho(x,T)||\rho_1(x)], \\
\text { s.t. } &\partial_{t} \rho +\nabla \cdot(\rho u)= \sigma^2 \Delta \rho ,\\
&\rho(x,0)=\rho_{0}(x). \\
\end{array}
\end{equation}

Note that the above formulation can be easily modified if the reference measure $\mathbb{Q}$ is given by other simple diffusion process. For example,  if $\mathbb{Q}$ is the measure generated by the Langevin diffusion stated in \eqref{problem:SOC}, it suffices to adjust the hypothetical velocity field $f$ by adding the given prior drift. we refer the readers to Sec \ref{sec:SOC} for further details and corresponding numerical experiments.

In above problem, KL divergence plays a soft constraint for terminal condition. As the regularization coefficient $\alpha$ goes to infinity, the minimizers of \eqref{eq:optm} should converge to the ones of classical Sch\"{o}dinger bridge problem. We will resolve this in section \ref{sec:Convergence} with the help of $\Gamma-$convergence.

\section{The CNF based machine learning framework}\label{sec:algorithm}
In this section, we propose a machine learning framework to solve the SBP via continuous normalizing flows and score matching algorithms. The training process of hypothetical velocity field is illustrated in section \ref{subsec:Training_hypothetical_velocity_field}. Additional implementation details and techniques are also discussed. Then, in section \ref{subsection:recovering_drift}, we recover the optimal drift based on the pre-trained hypothetical velocity field using classical score matching methods.

\subsection{Training hypothetical velocity field}\label{subsec:Training_hypothetical_velocity_field}
In our case, we aim to connect a real data distribution to the normal distribution via the Schr\"{o}ndinger Bridge. A neural network is employed to approximate the hypothetical velocity field. The network is trained to minimize a loss function that combines the KL divergence with the $L^2$ norm of the drift as a regularization term:
\begin{equation}
    \begin{aligned}
        \mathcal{J} &= \alpha \mathcal{J}_{\mathbb{KL}}+  \mathcal{J}_{B} \\
        \mathcal{J}_{\mathbb{KL}} &=\mathbb{KL}[\rho(z(x,T))\|\rho_1(z(x,T))] \\
         \mathcal{J}_{B} &=\int_{0}^{T}\int_{\mathbb{R}^d} \|f(z(x,t),t)-\nabla \log \rho(z(x,t),t)\|^2 \rho_0(x) dx dt
    \end{aligned}
\end{equation}
One can apply the same push-forward trick used in CNFs \cite{grathwohl2018ffjord} to simplify the KL term:
\begin{equation}
    \begin{aligned}
        & \mathbb{KL}\left[\rho(\boldsymbol{z}(\boldsymbol{x}, T)) \| \rho_1(\boldsymbol{z}(\boldsymbol{x}, T))\right] \\
        = &\int_{\mathbb{R}^d} \log \left(\frac{\rho(\boldsymbol{z}(\boldsymbol{x}, T))}{\rho_1(\boldsymbol{z}(\boldsymbol{x}, T))}\right) \rho(\boldsymbol{z}(\boldsymbol{x}, T)) \operatorname{det}(\nabla \boldsymbol{z}(\boldsymbol{x}, T)) \mathrm{d} \boldsymbol{x}, \\
        = &\int_{\mathbb{R}^d} \log \left(\frac{\rho_0(\boldsymbol{x})}{\rho_1(\boldsymbol{z}(\boldsymbol{x}, T)) \operatorname{det}(\nabla \boldsymbol{z}(\boldsymbol{x}, T))}\right) \rho_0(\boldsymbol{x}) \mathrm{d} \boldsymbol{x}, \\
        = &\int_{\mathbb{R}^d}\log \left(\rho_0(\boldsymbol{x})\right) \rho_0(\boldsymbol{x}) \mathrm{d} \boldsymbol{x}+\mathbb{E}_{\rho_{0}(x)} \left[  -\log \left(\rho_1(\boldsymbol{z}(\boldsymbol{x}, T))\right)-\log \operatorname{det}(\nabla \boldsymbol{z}(\boldsymbol{x}, T))\right] .
    \end{aligned}
\end{equation}
Note that the first part is a constant independent of the nerual network and can therefore be dropped during training. By substituting the normal distribution expression for $\rho_1$, we obtain the final loss function:
\begin{equation}\label{eq:LossSchB}
\begin{split}
    \mathcal{J}&=\mathbb{E}_{\rho_{0}(x)} \left[\alpha C(x,T)+B(x,T) \right],\\
    C(x,T)&=-\ell(x,T) +\frac{1}{2}|\boldsymbol{z}(\boldsymbol{x}, T)|^{2}+\frac{d}{2} \log (2 \pi),\\
    B(x,T)&=\frac{1}{2} \int_{0}^{T} |f(z(x, t), t)+ s(x,t)|^{2} d t,
\end{split}
\end{equation}
which can be approximated with the Monte Carlo approximation. Here, $\ell(x,t)$ and $s(x,t)$ are given in \eqref{eq:logdeter} and \eqref{eq:scorealongtraj}.  Algorithm \ref{alg:Solving_SBP_via_Normalizing_flow} presents our training process for the hypothetical velocity field.

\begin{algorithm}
\caption{Solving SBP via Normalizing flow}
\label{alg:Solving_SBP_via_Normalizing_flow}
\begin{algorithmic}[1]
\Require Data samples $\{ x_{i}\}_{i=1}^{N}$, time interval $[0,T]$, initializing network $f_{\theta}$
\For{number of training iterations}
    \State Solving ODE system \eqref{eq:original ODE} with $f_{\theta}$ to obtain $z(x_i,T)$, $\ell(x_i,T)$ and $s(x_i,t)$.
    \State Calculate cost function $\mathcal{J}=\alpha \mathcal{J}_{\mathbb{KL}}+ \mathcal{J}_{B} $
    \State Use ADAM optimizer optimizer to update network parameter of $f_{\theta}$
\EndFor
\end{algorithmic}
\end{algorithm}
In training, particularly for high-dimensional cases, two methods can be used to reduce computational costs of calculating $\nabla \cdot f$: random mini-batch method and unbiased linear-time log-density estimation.

\begin{itemize}
    \item {\bf Random mini-batch method} Inspired by the random batch method in stochastic gradient descent \cite{eon1998online} and interacting particle systems \cite{jin2020random}, one can select a mini-batch of indices to give an estimation of $\nabla \cdot f$.

    \item {\bf Unbiased linear-time log-density estimation} In general, for a vector-valued function $ f: \mathbb{R}^d \to \mathbb{R}^d $, computing the trace of the Jacobian $    \operatorname{Tr}(\nabla f) $ (or equivalently, the divergence $ \nabla \cdot f $) exactly incurs a computational cost of $\mathcal{O}(d^2)$. To avoid computing each entry of the diagonal of the Jacobian explicitly, we can leverage Hutchinson's trace estimator, as employed in CNFs \cite{grathwohl2018ffjord,finlay2020train}. This approach enables efficient approximation of $\nabla \cdot f$ when solving the second ODE in \eqref{eq:original ODE}, or when estimating data density during inference. The Hutchinson's trace estimator provides an unbiased estimate of the trace of a matrix \( A \) as follows:
    \begin{equation}
    \operatorname{Tr}(A) = \mathbb{E}_{p(\lambda)}\left[\lambda^T A \lambda\right],
    \end{equation}
    where $ \lambda $ is a $d$-dimensional random vector with zero mean $\mathbb{E}[\lambda] = 0 $ and identity covariance $\mathbf{Cov}(\lambda) = I$. Common choices for the distribution $ p(\lambda)$ include the standard Gaussian distribution or the Rademacher distribution. To compute $\operatorname{Tr}(\nabla f)$, we use automatic differentiation  to calculate  $\nabla (\lambda^T f)$ first. Then we get an unbiased estimate of trace by multiplying $\lambda$ again:
    \begin{equation*}
        \operatorname{Tr}(\nabla f) = \nabla \cdot f \approx  \nabla (\lambda^T f) \cdot \lambda .
    \end{equation*}

\end{itemize}

CNFs are trained to minimize the loss function in \eqref{eq:LossSchB}, which involves the solution to an initial value problem with an ODE parameterized by $\theta$. Chen et al. \cite{chen2018neural} employ the adjoint method to compute the gradient and subsequently perform backpropagation, which allows one to train the models with constant
memory cost as a function of depth. Typically, for any scalar loss function which operates on the solution to an initial value problem
\begin{equation}
L\left(\mathbf{z}\left(t_1\right)\right)=L\left(\int_{t_0}^{t_1} f(\mathbf{z}(t), t ; \theta) d t\right)
\end{equation}
then Pontryagin \cite{pontryagin2018mathematical} shows that its derivative takes the form of another initial value problem
\begin{equation}\label{ODE:adjoint_state}
\frac{d L}{d \theta}=-\int_{t_1}^{t_0}\left(\frac{\partial L}{\partial \mathbf{z}(t)}\right)^T \frac{\partial f(\mathbf{z}(t), t ; \theta)}{\partial \theta} d t
\end{equation}
The quantity  $\frac{\partial L}{\partial z(t)}$ is known as the adjoint state of the ODE. Chen et al. \cite{chen2018neural} proposed to use a black-box ODE solver to compute $z(t_1)$, and then another call to a solver to compute \eqref{ODE:adjoint_state} with the initial value $\frac{\partial L}{\partial z(t_1)}$. The adjoint method can be employed to solve high-dimensional cases of the SBP.

\subsection{Recovering the optimal drift}\label{subsection:recovering_drift}
It is important to note that we only have access to the score function along the trajectory. To recover the drift $u$, we can employ classical score matching algorithms to approximate the score function with a pre-trained hypothetical velocity field $f_\theta$. 

Given i.i.d. samples $\{x_1, x_2, \cdots, x_N \} \subset \mathbb{R}^{d}$ from an underlying distribution $\pi$, the classical score matching \cite{hyvarinen2005estimation} uses a neuron network $s_{\phi}(x) \in \mathbb{R}^d$ to approximate the score function $\nabla \log \pi(x)$ and minimize the $L^{2}$ loss function $\frac{1}{2} \mathbb{E}_{\pi}  ||s_{\phi}(x) - \nabla \log \pi(x)||_2^2 $, which can be shown equivalent to the following up to a constant
\begin{equation}
J_{SM}(\phi) \triangleq \mathbb{E}_{  \pi} \left[ \operatorname{tr} ( \nabla_x s_{\phi}(x) )+ \frac{1}{2} ||s_{\phi}(x)||_2^2 \right],
\end{equation}
After applying the Hutchinson's trace estimator mentioned above, the sliced score matching \cite{song2020sliced} minimize the following loss function:
\begin{equation}
J_{SSM}(\phi) \triangleq \mathbb{E}_{\pi} \mathbb{E}_{\rho(\lambda)}\left[ \lambda^{T}\nabla_{x} s_{\phi}(x)\lambda +\frac{1}{2} \left(\lambda^{T}s_{\phi}(x)\right)^{2}  \right],
\end{equation}
where $\nabla_{x} s_{\phi}$ denotes the Hessian of the score function. $\lambda \sim \rho_{\lambda}$ and $x \sim \pi$ are independent. For the sake of computational convenience, one may opt to model $\rho_{\lambda}$ as either a multivariate Rademacher distribution or a multivariate normal distribution since $\mathbf{E}(\lambda)=0$ and $\mathbf{Cov}(\lambda) = I$

In our case, to recover the drift, we can parameterize it as a neural network $u_{\phi}(x,t)$ and expect the function 
\begin{equation*}
    s_{\phi}(x,t):= \frac{u_{\phi}(x,t)-f_{\theta}(x,t) }{\sigma^2}
\end{equation*}
to approximate the ground truth score $\nabla \log \rho_t(x)$. Note that the score function $s_{\phi}(x,t)$ here is different from that in equation \eqref{eq:scorealongtraj} as the $x$ variable here indicates a general location at current time point instead of the initial location for a trajectory. Similar to Noise Conditional Score Networks (NCSNs) \cite{song2019generative} and diffusion models \cite{ho2020denoising}, the joint score matching objective is given as follows
\begin{equation}
J_{JSM}(\phi) \triangleq \mathbb{E}_{t \sim \mathcal{U}(0,1)} \mathbb{E}_{x_t \sim \rho_t} \mathbb{E}_{ \lambda \sim\rho(\lambda)} w(t) \left[ \lambda^{T}  \nabla_{x} s_{\phi}(x_t,t)  \lambda+\frac{1}{2} \left(\lambda^{T}s_{\phi}(x_t,t)\right)^{2}  \right],
\end{equation}
where $w(t)$ is a predefined weight function. Thus as long as we obtain a pre-trained hypothetical velocity field $f_{\theta}$, we can sample from standard Gaussian and generate trajectories, which will be used in the training of score matching to obtain drift function $u_{\phi}$. Algorithm \ref{alg:Drift_recovering} presents the process of recovering optimal control via score matching. We remark that if $\sigma$ is very small, one may want to train $s_{\phi}$ as an independent neural network and then recover $u_{\phi}=f_{\theta}+\sigma^2s_{\phi}$ for better stability.

\begin{algorithm}
\caption{Drift recovering in SBP}
\label{alg:Drift_recovering}
\begin{algorithmic}[1]
\Require pre-trained hypothetical velocity $f_{\theta}$, a time discretization $0= t_0 < t_1 < \cdots < t_L=T$.
\State Sample from Gaussian distribution and use $f_{\theta}$ to generate trajectories $\{ x_i^{t_k}\}_{i=1}^{N}$ by solving the first equation of the ODE system \eqref{eq:original ODE} backward in time.
\For{number of training iterations}
    \State select a batch of sample pairs $ (x_{i}^{t_k},t_k)$ from trajectories
    \State Calculate cost function $J_{JSM}$
    \State Minimizing $J_{JSM}(\phi)$  with ADAM  optimizer 
\EndFor
\State \Return Optimal drift $u_{\phi}$
\end{algorithmic}
\end{algorithm}

\section{Convergence of the optimization problem to SBP}\label{sec:Convergence}

We introduced a relaxation factor $\alpha$ for the terminal distribution in \eqref{eq:LossSchB} to render the loss function computationally tractable. We now rigorously demonstrate that the solution to our relaxed optimization problem converges to that of the original Schr\"odinger bridge problem as $\alpha$ approaches infinity by leveraging the framework of $\Gamma$-convergence, as outlined earlier. See similar analysis of the convergence from OT-Flow to optimal transport (OT) in \cite{jing2024convergence}.

For the convenience, we will assume that $\Omega\subset\R^d$ is a bounded domain with smooth boundary in this section.

\subsection{Mathematical formulation of the problems}

Following \cite{santambrogio_benamou-brenier_2015}, we make use of the variables $(\rho, m)$ where $m=\rho u$ so that there is better convexity. Then the optimization problem of SBP \eqref{eq:schb} can be reformulated as 
\begin{equation}\label{eq:schb-E}
\begin{array}{ll}
\min\limits_{\rho, m} & \int_{0}^{T} \int_{\mathbb{R}^d} \frac{|m|^2}{2\rho} dx dt+ \chi_{E}^\infty\\
\text { s.t. } &\partial_{t} \rho +\nabla \cdot m= \sigma^2 \Delta \rho ,\quad x\in \Omega\\
& (m-\sigma^2\nabla\rho)\cdot n=0, \quad x\in \partial\Omega\\
&\rho(x,0)=\rho_{0}(x),\quad x\in \Omega
\end{array}
\end{equation}
where $E$ is the set of the terminal constraints $\rho(x,T) = \rho_1(x)$ and $\chi_{E}^\infty$ is the indicator function, 
\begin{equation}
\chi_{E}^\infty = \begin{cases}
0, &\text{if } x \in E, \\
+\infty, & \text{otherwise.}
\end{cases}
\end{equation}. 
The corresponding optimization problem of our algorithm \eqref{eq:optm} is given by
\begin{equation}\label{eq:optm-E}
\begin{array}{ll}
\min\limits_{\rho, m} & \int_{0}^{T} \int_{\mathbb{R}^d} \frac{|m|^2}{2\rho} dx dt +\alpha \mathbb{KL}[\rho(x,T)||\rho_1(x)], \\
\text { s.t. } &\partial_{t} \rho +\nabla \cdot m= \sigma^2 \Delta \rho ,\quad x\in \Omega\\
& (m-\sigma^2\nabla\rho)\cdot n=0, \quad x\in \partial\Omega\\
&\rho(x,0)=\rho_{0}(x),\quad x\in \Omega.
\end{array}
\end{equation}
In general, $\rho$ and $m$ are measures and $\rho(\cdot, T)$ may not be well-defined either before the time regularity in time has been verified. Hence, we need to put the above problem into a rigorous framework.

To proceed, we first specify the topological space and associated constraints. The topology spaces we will work on are 
\[
X:=L^1([0, T]; \mathcal{P}(\Omega)), \quad  Y:= L^1([0, 1]; \mathcal{M}(\Omega)^d).
\]
Here, $\mathcal{P}(\Omega)$ is the set of probability measures on $\Omega$ and $\mathcal{M}(\Omega)^d$ is the set of $\R^d$-valued Radon measures on $\Omega$.  Then we equip the product space $X \times Y$ for $(\rho, m)$ with the product weak topology, i.e.,  $(\rho_n, m_n)\to (\rho, m)$ means: for all $f \in C_b([0, T]\times \Omega; \R)$ and $g\in C_b([0, T]\times \Omega; \R^d)$, one has
\begin{equation*}
     \int_0^T \int_{\mathbb{R}^d} f d\rho_n + \int_0^T\int_{\mathbb{R}^d} g dm_n \to \int_0^T \int_{\mathbb{R}^d} f d\rho +\int_0^T\int_{\mathbb{R}^d} g dm.
\end{equation*}
Obviously $X \times Y$ is closed and first-countable. Then the constraint $\partial_t\rho + \nabla\cdot m=\sigma^2\Delta\rho$ introduces a subspace $\mathcal{H}$ of $X$:
\begin{multline}
    \mathcal{H}:=\Bigg\{ (\rho,m) \in X \times Y: -\int_0^T \left( \int_{\Omega}(\partial_t \varphi) \rho (dx) +\nabla \varphi \cdot m(dx) \right) dt -\int_{\Omega}\varphi(x,0)\rho_{0}(dx) \\
     =\sigma^2 \int_0^T\int_\Omega\Delta\varphi\cdot\rho(dx)\,dt, \forall \varphi \in C_{b}^{2}([0,T]\times \Omega), 
     \left.\frac{\partial\varphi}{\partial n}\right\vert_{\partial\Omega}=0, \varphi(\cdot,T)=0 \Bigg\}.
\end{multline}
The subspace $\mathcal{H}$ contains the boundary conditions implicitly and naturally. This space is quite large for $\rho$, compared to the original requirement $\rho\in \mathcal{P}(C([0, T];\Omega))$. However, as we will see later, this space is more convenient for analysis and it will not include more solutions. Clearly $\mathcal{H}$ is closed due to the fact that the constraints are linear.

Similarly, $E$ can now be rigorously defined as follows:
\begin{multline}
    E:=\Bigg\{ (\rho,m) \in X \times Y: -\int_0^T \left( \int_{\Omega}(\partial_t \varphi) \rho (dx) +\nabla \varphi \cdot m(dx) \right) dt +\int_\Omega \varphi(x,T)\rho_1 (dx) \\
    -\int_{\Omega}\varphi(x,0)\rho_{0}(dx)  = \sigma^2 \int_0^T\int_\Omega\Delta\phi\cdot\rho(dx)\,dt, \forall \varphi \in C_{b}^{2}(\bar{\Omega} \times [0,T]), \left.\frac{\partial\varphi}{\partial n}\right\vert_{\partial\Omega}=0 \Bigg\}.
\end{multline}
One should note that $E$ is a closed subspace of $\mathcal{H}$, and this is the rigorous definition of the set $E$ introduced in \eqref{eq:schb-E}.

Next, we treat the issue for $\rho(\cdot, T)$. We following the approach in \cite{jing2024convergence} to resolve this. We consider the following 
\begin{equation}
    \bar{\rho}(\cdot, T) :=\lim_{\delta\to 0}\frac{1}{\delta}\int_{T-\delta}^T \rho_s\,ds.
\end{equation}
Clearly, if $t\mapsto \rho_t$ has a version that is continuous at $T$, $\bar{\rho}(\cdot, T)$ is well-defined.
With this, we define the functional
\begin{equation}
    G(\rho,\rho_1):=\left\{\begin{array}{lr}
    \ \mathbb{KL}\left[\bar{\rho}(x,T)\|\rho_1(x)\right], & \text { if } \bar{\rho}(\cdot, T)~\text{exists.}\\
    \ \infty , & \text { else. }
    \end{array}\right.
\end{equation}
$G(\rho,\rho_1)$ actually corresponds to the KL divergence term if we later focus on the feasible points, for which $\rho$ is actually continuous in time.

We now can define the functionals $F_\alpha$ and $F_\infty: \mathcal{H} \rightarrow \mathbb{R} \cup {\infty}$ corresponding to \eqref{eq:optm-E} and  \eqref{eq:schb-E} respectively as:
\begin{equation}\label{eq:functionals}
\begin{aligned}
F_\alpha(\rho, m)&=\int_{0}^{T} \mathscr{B}_2(\rho_t, m_t)\,dt +\alpha G(\rho,\rho_1),\\
F_\infty(\rho, m)&=\int_{0}^{T} \mathscr{B}_2(\rho_t, m_t)\,dt+ \chi_{E}^\infty.
\end{aligned}
\end{equation}
With Proposition \ref{prop:benamou}, optimization for the functional $F_{\alpha}$ and $F_{\infty}$ are then the rigorous definitions of the optimization problems \eqref{eq:optm-E} and \eqref{eq:schb-E}, respectively.
Similar as shown in \cite{jing2024convergence}, they are lower semi-continuous with respect to the topology considered.

\subsection{Regularity of feasible points}

It is well-known that the minimizer of $F_\infty$ is the solution of the Schr\"{o}dinger Bridge Problem, and the existence of solution to SBP is a well-established result, i.e. in \cite{christian_leonard_survey_2014}. Using the solutions, it is to construct a feasible point of $F_\alpha$.
Hence, the feasible points of $F_\alpha$ and $F_\infty$ actually exists. 

We show that the feasible points of $F_\alpha$ and $F_\infty$ have good time regularity properties by the following proposition.
\begin{proposition}\label{prop:c1}
If $(\rho, m)$ is a feasible solution of \eqref{eq:functionals} (for $F_\alpha$ or $F_\infty$), $m\ll \rho$ and the Radon-Nikodym derivative $u=\frac{dm}{d\rho}$ is in $L^1([0, T]; L^2(\rho_t))$, and there is a version of $\rho$ such that $t\mapsto \rho_t$ is continuous in $\mathcal{W}_2(\Omega)$, where $\mathcal{W}_{2}:=\{\mu \in \mathcal{P}(\Omega)|\int |x|^{2} \mu(dx)< \infty\}$.  Moreover, if the initial entropy is finite: $H(\rho_0):=\int \rho_0\log \rho_0\,dx<\infty$, then $\rho$ is absolutely continuous in $\mathcal{W}_2(\Omega)$.  
\end{proposition}

\begin{proof}
If $(\rho, m)$ is a feasible solution, then
\begin{equation*}
    \int_0^T\mathscr{B}_2(\rho_t, m_t)\,dt<+\infty,
\end{equation*}
By Proposition \ref{prop:benamou}, $m\ll \rho$ for a.e. $t$ and 
$ u_t=\frac{dm_t}{d\rho_t}$
satisfies 
\begin{equation*}
\int_0^T \|u\|_{L^2(\rho_t)}^2\,dt=\int_0^T \|u_t\|^2\rho(dx) = 2\int_0^T \mathscr{B}_2(\rho_t, m_t)\,dt<+\infty.
\end{equation*}
so that $u\in L^1([0, T]; L^2(\rho_t))$.

Since $\Omega$ is bounded, the continuity in $\mathcal{W}_2$ is equivalent to the weak continuity, which follows from the fact that
$\rho$ is a weak solution to $\partial_t\rho+\nabla\cdot(\rho u)=0$ as it is in $\mathcal{H}$.
In fact, we take $\varphi(x, t)=\phi(x)h(t)$ such that $h(T)=0$, $\phi\in C_b^2$ and $\frac{\partial\phi}{\partial n}=0$, one has
\[
-\int_0^T h'(t) \int_{\Omega}\phi \rho_t(dx)\,dt-\int_0^Th(t)\int\nabla\phi \cdot u \rho(dx)\,dt
-h(0)\int\phi(x)\rho_0(dx)=\sigma^2\int_0^Th(t)\int \Delta\phi \rho(dx)\,dt.
\]
This holds for any $h\in C^1[0, T]$ with $h(T)=0$. This means that
$t\mapsto g(t):=\int \phi\rho_t(dx)$ has a weak derivative and the weak derivative is given by
\[
\frac{d}{dt}\int \phi\rho_t(dx)=\int \nabla\phi\cdot u \rho(dx)+\sigma^2\int\Delta\phi\rho(dx).
\]
Since $u\in L^2(0, T; L^2(\rho_t))$, the right hand side is integrable in $(0, T)$.
Hence, $t\mapsto g(t)$ is absolutely continuous in $(0, T)$. Moreover, 
by the equation above, one has $g(0+)=\int \phi(x)\rho_0(dx)$. 
This indicates that
\[
\int \phi\rho_{t+h}(dx)
\to \int \phi \rho_t(dx),
\quad h\to 0,
\]
for any $\phi\in C_b^2(\Omega)$. By a standard density argument, the test function can be generalized to $C_b$ class and thus $t\mapsto \rho_t$ is weakly continuous and thus continuous in $\mathcal{W}_2(\Omega)$, due to the boundedness of $\Omega$.


For the absolute continuity of $\rho$, recall
\begin{equation*}
f_t = u_t - \sigma^2\nabla\log\rho_t.
\end{equation*}
Then, $\rho_t$ satisfies the continuity equation with this equivalent drift $f_t$.
By the standard result in optimal transport theory \cite[Theorem 5.14]{santambrogio_benamou-brenier_2015},
if $\int_0^T\|f_t\|_{L^2(\rho_t)}^2\,dt<\infty$, then $\rho_t$ is absolutely continuous in $\mathcal{W}_2$. 
By the integrability of $u_t$, it suffices to verify $\int_0^T\|\nabla\log\rho_t\|_{L^2(\rho_t)}^2\,dt<\infty$.
Consider the entropy:
\begin{equation*}
H(\rho_t) = \int_\Omega\rho_t\log\rho_t\,dx.
\end{equation*}
The formal calculation implies that
\begin{multline*}
\frac{d}{dt}H(\rho_t) = \int_\Omega \rho_t v_t\cdot\nabla\log\rho_t\,dx  -\sigma^2\|\nabla\log\rho_t\|_{L^2(\rho_t)}^2 \\
\leq \|u_t\|_{L^2(\rho_t)}\|\nabla\log\rho_t\|_{L^2(\rho_t)}  -\sigma^2\|\nabla\log\rho_t\|_{L^2(\rho_t)}^2
\le C\|u_t\|_{L^2(\rho_t)}^2  -\frac{1}{2}\sigma^2\|\nabla\log\rho_t\|_{L^2(\rho_t)}^2.
\end{multline*}
On bounded domain, $H(\rho_t)$ is always bounded from below. Integrating over $[0, T]$ with $\int_0^T\|u_t\|_{L^2(\rho_t)}\,dt<+\infty$, combining the boundedness of $H(\rho_0)$, one has 
\[
\int_0^T\int_{\Omega}{\lvert\nabla\log\rho_t\rvert}^2\rho_t(dx)\,dt<+\infty.
\]
This calculation can be made rigorous by mollification and taking the mollification parameter to zero. 
 Hence $\int_0^T\|f_t\|_{L^2(\rho_t)}^2\,dt$ is bounded, and $\rho_t$ is absolutely continuous.
\end{proof}

According to this proposition, $\rho_t$ is continuous and thus 
\[
G(\rho, \rho_1)=\mathbb{KL}\left[\rho(x,T)\|\rho_1(x)\right].
\]
That means the functional we considered indeed is the one we desire.

\subsection{Convergence of the functionals}

We now establish the result concerning the $\Gamma$-convergence and the behavior of the minimizers.

\begin{theorem}
Assume that $F_\alpha$(resp. $F_\infty$) has at least one feasible point over $\mathcal{H}$, then $F_\alpha$(resp. $F_\infty$) has a unique global minimizer over $\mathcal{H}$. 
Moreover, $F_\alpha$ $\Gamma$-converges to $F_\infty$, and for any sequence $\alpha_I\to\infty$, the minimizers of $F_{\alpha_I}$ have a subsequence converging to the minimizers of $F_\infty$.
\end{theorem}
\begin{proof}
Since $F_\alpha$ is lower-semicontinuous (by Proposition~\ref{prop:benamou}), and $F_\alpha\uparrow F_\infty$ pointwisely, one concludes that by Proposition \ref{pro:gammaconv} that $F_\alpha$ $\Gamma$-converges to $F_\infty$. 

By Proposition \ref{prop:benamou}, we know that $\mathscr{B}_2(\rho, m)\geq 0$, hence $F_\alpha$ and $F_\infty$ are bounded from below. Then there exists $F_\alpha^*\in[0, +\infty)$, s.t. $F_\alpha^*=\inf_{(\rho, m)\in\mathcal{H}} F_\alpha(\rho, m)$.

Consider a feasible minimizing sequence $(\rho_n, m_n)$ such that $F_\alpha(\rho_n, m_n)\to F_\alpha^*$. Then by the H\"{o}lder inequality, we have
\begin{multline}
\|m_n\| = \int_0^T\int_{\mathbb{R}^d}|m_n(dx, t)|\,dt\leq{\left(2\int_0^T\int_{\mathbb{R}^d}\rho_n(dx) dt\int_0^T\mathscr{B}_2(\rho_n, m_n) dt\right)}^{1/2}\\
=\sqrt{2T}{\left(\int_0^T\mathscr{B}_2(\rho_n, m_n) dt\right)}^{1/2}\leq\sqrt{2T F_\alpha(\rho_n, m_n)},
\end{multline}
which implies that $\sup_n\|m_n\|<+\infty$, and further $\sup_n(\|m_n\|+\|\rho_n\|)<+\infty$. According to the Banach-Alaoglu theorem, there exists a subsequence $(\rho_{n_k}, m_{n_k})$ converging weakly to $(\rho^*, m^*)$ in $\mathcal{H}$. Together with the lower-semicontinuity of $F_\alpha$, we have
\begin{equation}
F_\alpha^*\leq F_\alpha(\rho^*, m^*)\leq\liminf_{k\to\infty}F_\alpha(\rho_{n_k}, m_{n_k})=F_\alpha^*.
\end{equation}
Hence the minimizer of $F_\alpha$ exists. Similarly, the minimizer of $F_\infty$ exists.

Moreover, since $F_{\alpha}$ is strictly convex (by the strict convexity of the Benamou-Brenier functional), the minimizer must be unique.

Let $(\rho^\alpha, m^\alpha)$ be the minimizer of $F_\alpha$, and $(\rho^\infty, m^\infty)$ be the minimizer of $F_\infty$. It's clear that
\begin{equation}
F_\alpha(\rho^\alpha, m^\alpha)\leq F_\alpha(\rho^\infty, m^\infty)\leq F_\infty(\rho^\infty, m^\infty),
\end{equation}
then
\begin{equation}
\sup_\alpha(\|m^\alpha\|+\|\rho^\alpha\|)\leq \sqrt{2TF_\infty(\rho^\infty, m^\infty)}+T<+\infty.
\end{equation}
Again by the Banach-Alaoglu theorem, there exists a subsequence $(\rho^{\alpha_k}, m^{\alpha_k})$ converging weakly to $(\rho, m)$ in $\mathcal{H}$. Since $F_\alpha$ $\Gamma$-converges to $F_\infty$, $(\rho, m)$ is a minimizer of $F_\infty$.

\end{proof}

Lastly, we comment that the original problem requires $\rho\in \mathcal{P}(C([0, T];\Omega))$, while we seek $\rho$ in a much larger space. However, due to the uniqueness of the minimizer, if the original Schr\"odinger bridge problem has a solution, it must be the solution in our framework. Hence, the mathematical formulation here is fine.

\section{Numerical experiments}\label{sec:numericalexp}
In this section, we consider some numerical examples to demonstrate our deep learning framework for solving the SBP. Our code is available at \href{https://github.com/sharkjingyang/SchB}{https://github.com/sharkjingyang/SchB}. In low dimensional experiments, we simply adopted the MLP framework, since this architecture already possesses sufficient approximation capability to represent the hypothetical velocity field in low dimensional cases. High order adaptive ODE solver is employed and an implementation of the adjoint method is utilized for back-propagation. Meanwhile we ensure tolerance is set low enough so numerical error from ODE solver can be negligible.

\subsection*{Experiment setup}
    
    \noindent \textbf{Hyper-parameters.} In all experiments, we consider the time interval as $[0,1]$, i.e. $T=1$. We take $\alpha=10$ as a default choice to make a balance for terms in the cost function.
    \\
    \noindent \textbf{Architectures} We use sine and cosine functions for time embedding, which are then processed through an MLP. The resulting embeddings are concatenated with the input data. The network architecture consists of three ConcatSquash linear layers with tanh activation functions.
    \\
    \noindent \textbf{Optimization.} For stochastic gradient descent, we choose Adam \cite{kingma2014adam} with a learning rate of 1e-3.

\subsection{1D Gaussian mixture}
As a first example, we employ the 1-D Gaussian mixture problem to evaluate the performance of our algorithm in solving the SBP, including verifying convergence properties and observing typical behaviors. We sample from the Gaussian mixture
\begin{gather}
\rho_0(x)= \frac{1}{2}\cdot\frac{1}{\sqrt{2\pi}}e^{-(x+3)^2/2}+\frac{1}{2}\cdot\frac{1}{\sqrt{2\pi}}e^{-(x-3)^2/2}
\end{gather}
and use our framework to learn the SB dynamics from $\rho_0$ to standard normal distribution $\rho_1$. Figure \ref{rho_0_1d} illustrates the forward process from a Gaussian mixture to a normal distribution, presented in histogram form. With a pre-trained hypothetical velocity field $f_{\theta}$, one can sample trajectories and employ score matching algorithms to approximate the score function. As a result, the optimal drift can be recovered as well. Figure \ref{plot:ODE_sample&SDE_sampler} shows ODE trajectories and SDE paths between two marginal distributions. 

\begin{figure}[htbp]
\includegraphics[width=13cm]{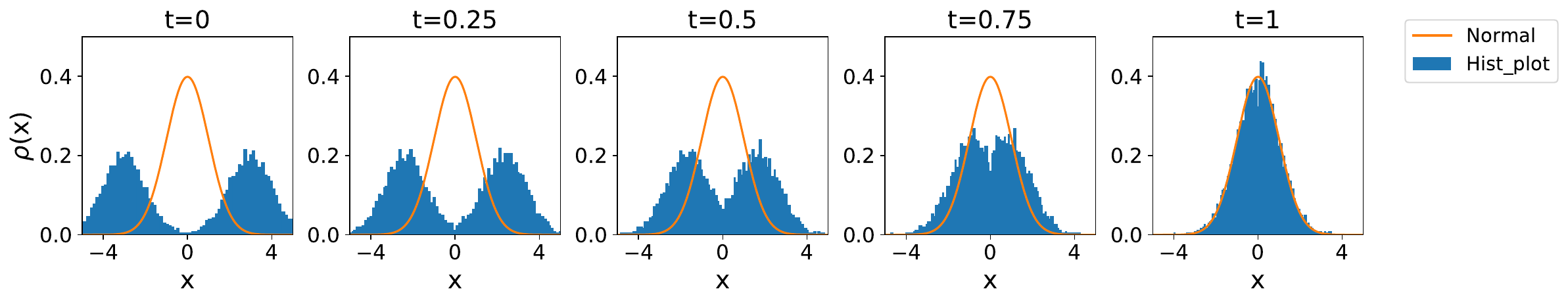}
\centering
\caption{Evolution of particle distribution transitioning from Gaussian mixture to Standard Gaussian.
}
\label{rho_0_1d}
\end{figure}

\begin{figure}[htbp]
\vspace{-0.5cm}
\centering
\subfigure[ODE trajectories for $dX= f_{\theta} dt$]{
\begin{minipage}[t]{0.48\linewidth}
\centering
\includegraphics[width=6.7cm]{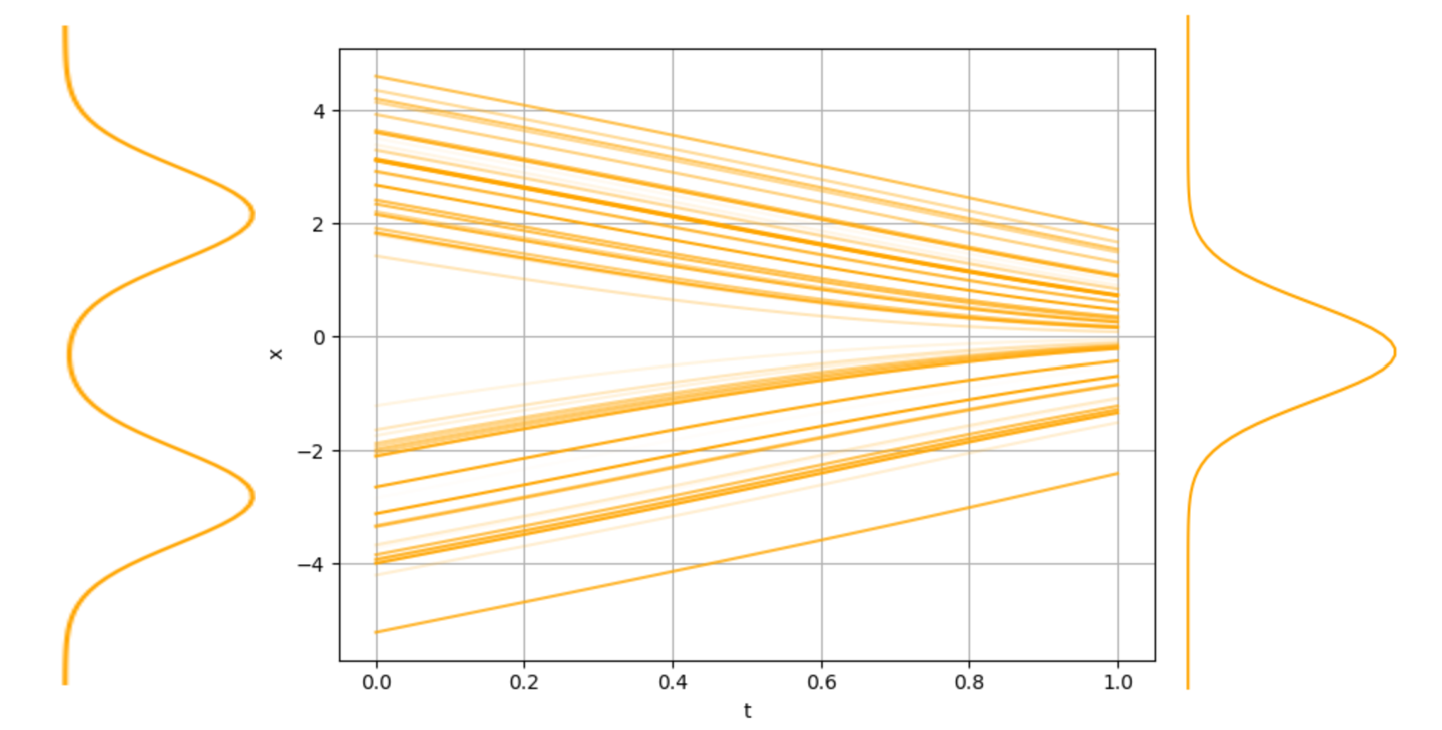}
\end{minipage}
}
\subfigure[SDE paths for $ dX= v dt + \sqrt{2} dW$]{
\begin{minipage}[t]{0.48\linewidth}
\centering
\includegraphics[width=6.7cm]{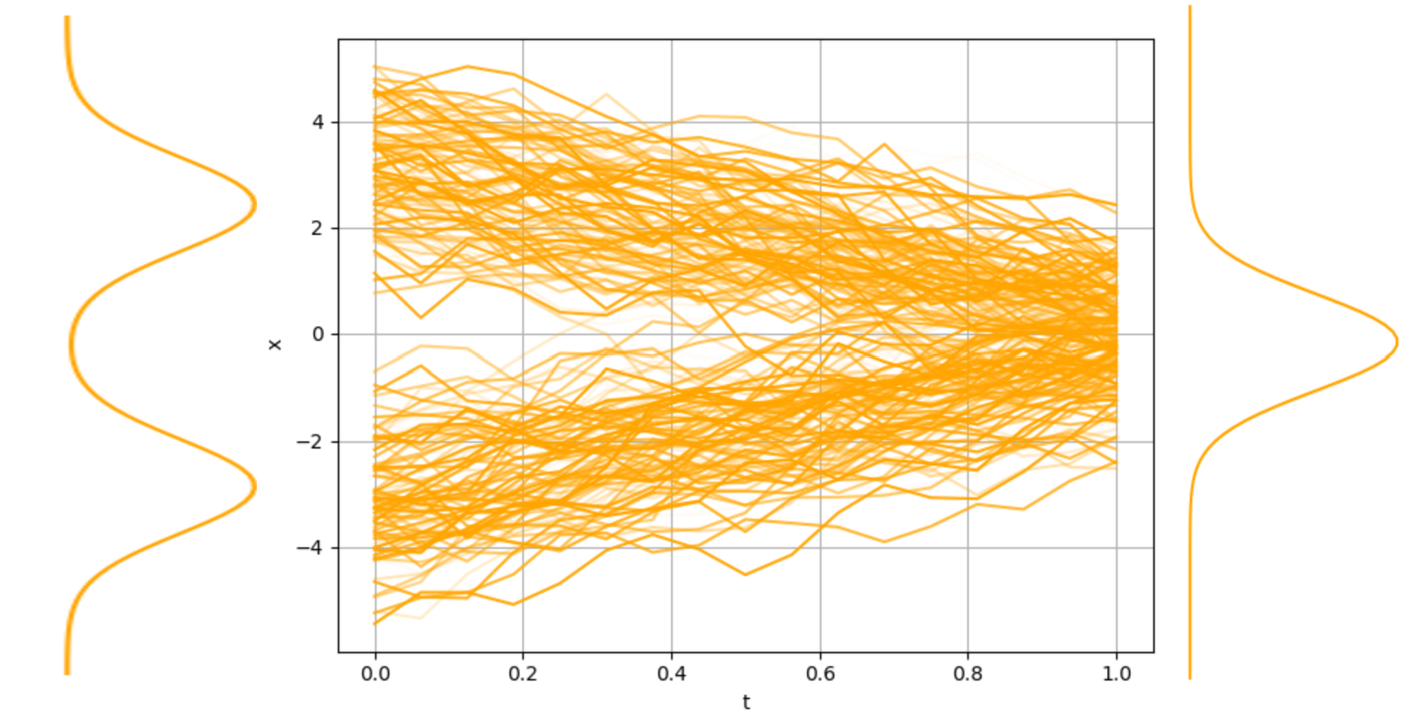}
\end{minipage}%
}
\centering
\caption{ODE interpolation trajectories  and SDE interpolation paths.}
\label{plot:ODE_sample&SDE_sampler}
\end{figure}


We also compare our results with a traditional optimization method. We employ a similar numerical scheme described in \cite{jing2024machine} to transform the optimization problem into discrete form and use a primal-dual hybrid algorithm to solve the SBP on the grid. The details of numerical implementation of the primal-dual algorithm can be found in Appendix \ref{Appendix:primal_dual}. In the 1D case, the minimizer of the discrete SBP optimization problem serves as a reference solution to verify the accuracy of our algorithm. Figure \ref{plot:score} displays the score function $s_{\phi}$ obtained through score matching, compared to the reference score, which is computed from the solutions of the discrete optimization problem \eqref{eq:discrete_minmax} using finite differences. The consistency of the score functions suggests that our algorithm provides a good approximation of the time marginals, as the trajectories used for training score matching algorithms are generated by the well-trained hypothetical velocity field $f_{\theta}$. Figure \ref{plot:compare_PD_v} shows the SBP's optimal drift recovered after score matching as suggested in section \ref{subsection:recovering_drift} and the reference drift. It can be observed that the two drifts match well in the high-density region, particularly within $\left[-3,3\right]$. The reference drift exhibits  oscillations near the boundary due to the boundary conditions imposed in the discrete Schr\"{o}dinger optimization problem, whereas our SBP-solving algorithm via normalizing flows does not strictly enforce a bounded domain or mass conservation.

We plot the hypothetical velocity field $f$ with different noise scale $\sigma$ solved by our algorithm, and the velocity field from optimal transport as computed by OT-Flow as show in Figure \ref{plot:compare_OT_v_drift}. A higher noise scale corresponds to a larger velocity in the initial stage, which aligns with intuitive understanding. In our 1D experiments, the hypothetical velocity field $f$ in SBP is quite similar compared with the optimal velocity field in OT. A plausible explanation is that the SBP can be interpreted as a regularized entropy optimal transport problem. As the noise scale diminishes to zero, the minimizers of SBP are anticipated to converge to those of optimal transport (OT). 

\begin{figure}[htbp]
\includegraphics[width=13cm]{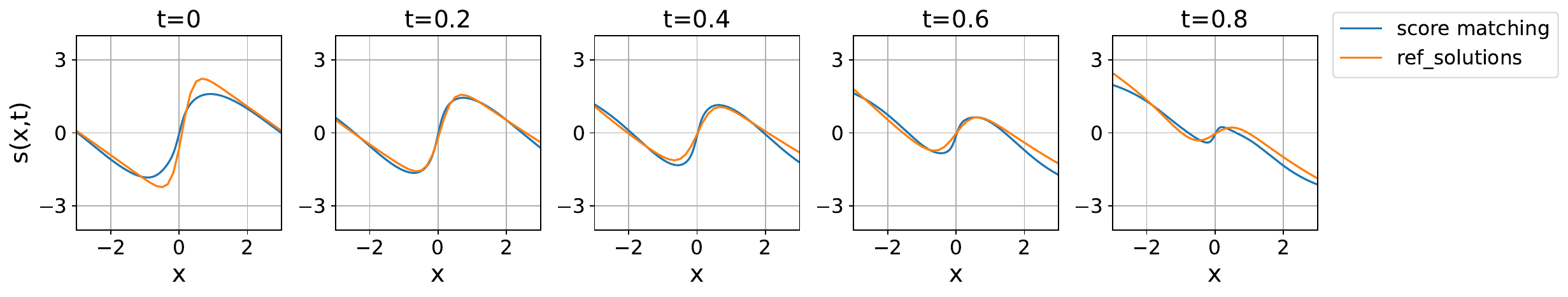}
\centering
\caption{Comparison of score function between score matching solutions and reference ones.
}
\label{plot:score}
\end{figure}

\begin{figure}[htbp]
\includegraphics[width=13cm]{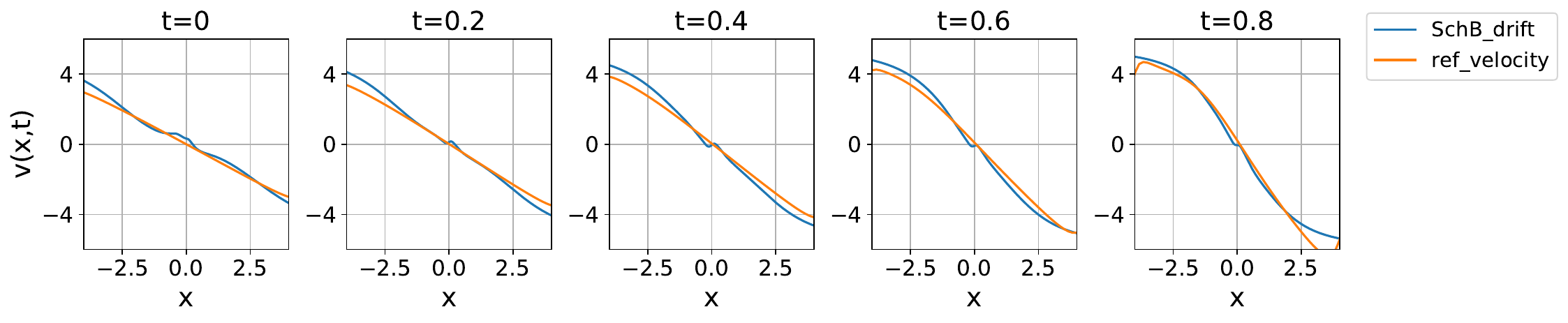}
\centering
\caption{Comparison of drift of solutions between our algorithm and tradition primal dual method.
}
\label{plot:compare_PD_v}
\end{figure}

\begin{figure}[htbp]
\includegraphics[width=13cm]{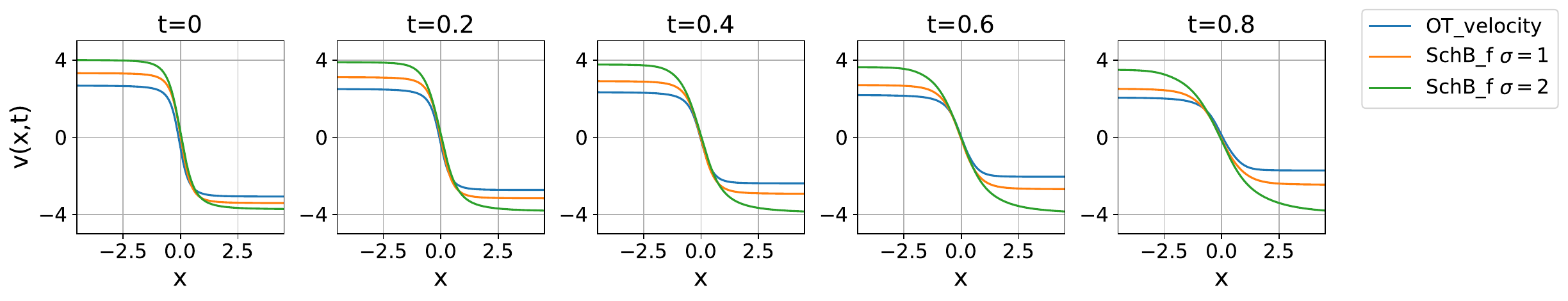}
\centering
\caption{Comparison of SBP hypothetical velocity field with different $\sigma$ and the velocity field of optimal transport problem. 
}
\label{plot:compare_OT_v_drift}
\end{figure}

During training, we set a relatively small $\alpha$ as suggested by the $\Gamma$-convergence analysis. In this small $\alpha$ regime, as the training progresses, one can expect $\mathbb{KL} [\rho(x,T)||\rho_1(x)] \approx 0$. To verify the convergence, we plot the curve of the optimization target for different values of $\alpha$ in Figure \ref{plot:gamma_convergence}. Each point on the curve represents the average value obtained from five independent experiments, where the settings are the same except for the value of $\alpha$. As $\alpha$ increases, we observe that $\mathcal{J}_{B}$ has a slower increasing rate, which is expected to eventually converge the minimum of the SBP. The optimization target computed by the traditional optimization method is about 3.493.


\subsection{2-D toy problems} 

In this subsection, we visualize the generative performance on classical 2-D synthetic examples to evaluate the potential of our algorithm as a generative model. We use trained hypothetical velocity field, or alternatively SB drift with diffusion process to simulate an inverse transformation (sampling from 2-D standard normal distribution and pushing particles back to the data distribution), from which one can compare the similarity of original data with generative distribution. The high similarity indicates that our model can generate samples to approximate $\rho_0$ with satisfactory accuracy even though $\rho_0$ has separate supports. Numerical results are illustrated in Figure \ref{plot:2d_moon_8g_checker}, which demonstrates the potential of our model to function as a generative model.

\begin{figure}[htbp]
    \vspace{-0.5cm}
    \centering
    \begin{minipage}{0.43\textwidth}
        \centering
        \includegraphics[width=\linewidth]{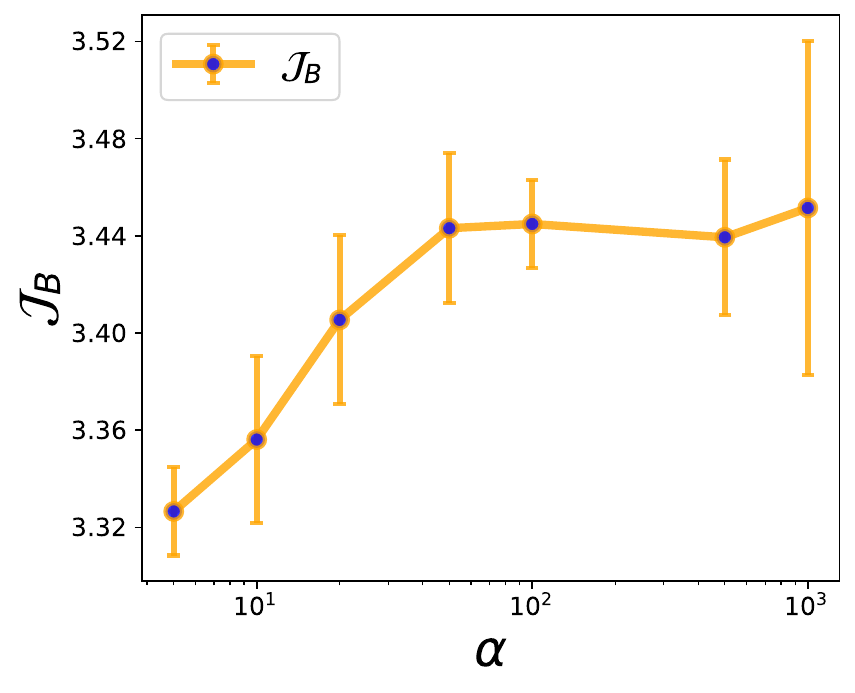}
        \caption{$\mathcal{J}_{B}$ with different $\alpha$}
        \label{plot:gamma_convergence}
    \end{minipage}
    \hfill
    \begin{minipage}{0.55\textwidth}
        \centering
        \includegraphics[width=\linewidth]{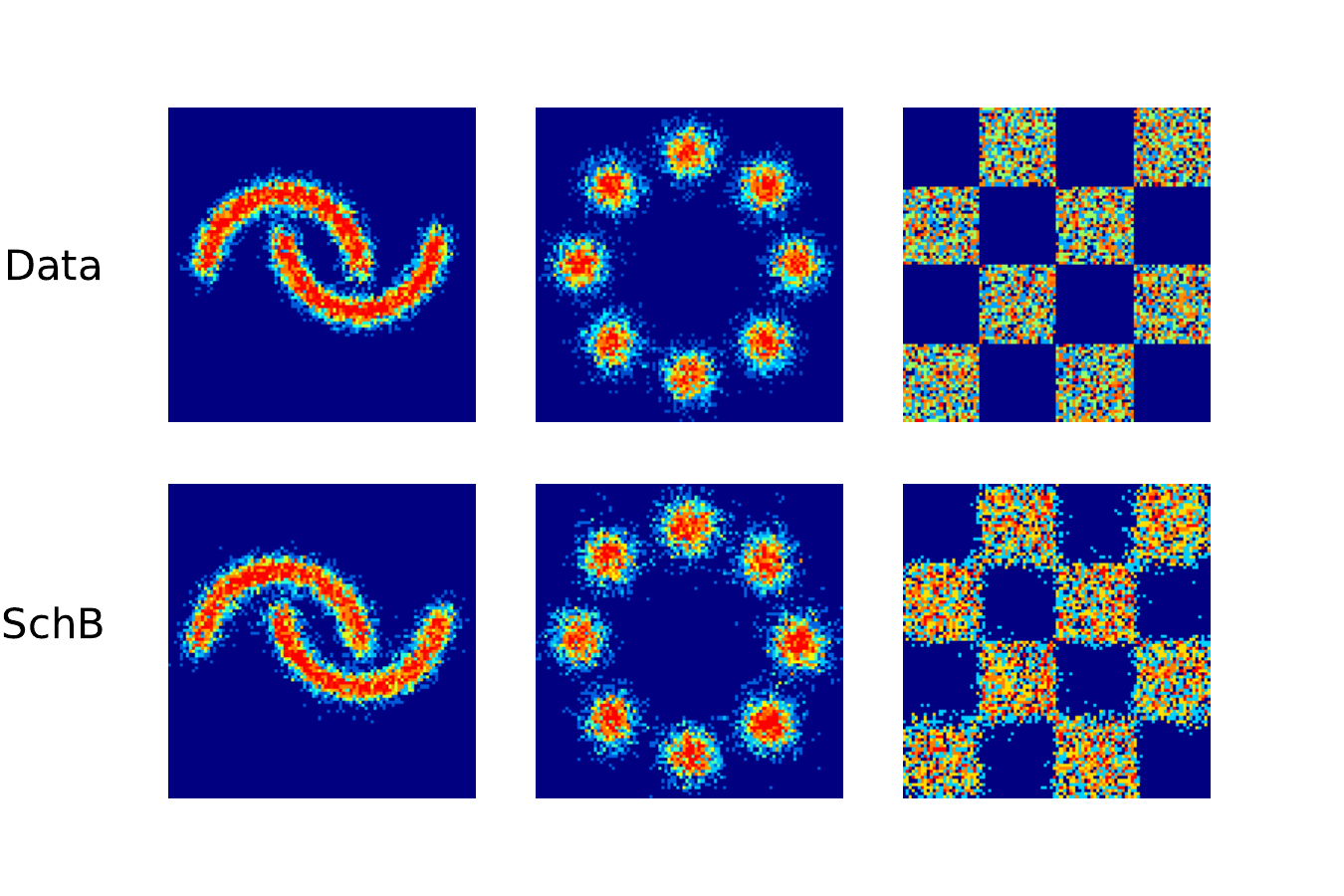}
        \caption{Comparison of data and generated samples}
        \label{plot:2d_moon_8g_checker}
    \end{minipage}
\end{figure}

The top two rows of Figure \ref{plot:2d_CNFs_vs_SchB} compare the forward and generation processes of the SchB, demonstrating the invertibility of the mappings and the stability of the numerical solvers. The third row presents the generation process of classical CNFs. In comparing the generation processes of SchB and classical CNFs, we observe that classical CNFs do not follow a 'straight' path, whereas the SchB corresponds to deblurring in the later stages of generation.

\begin{figure}[htbp]
\includegraphics[width=12cm]{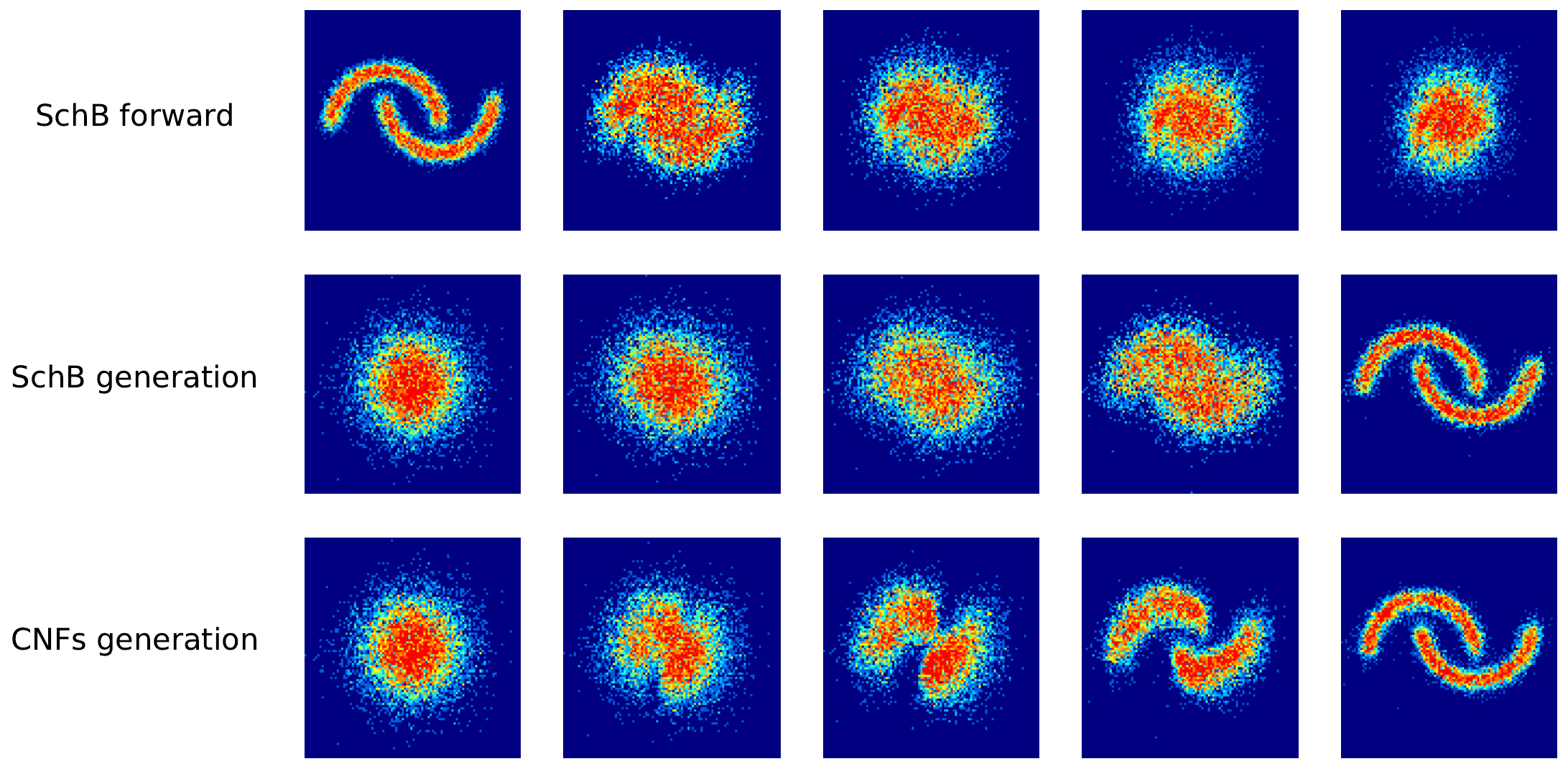}
\centering
\caption{Top: Forward process from data of moons to the normal distribution. Middle: Generation process of SchB from Gaussian samples to new data samples. Bottom: Generation process of classical CNFs from Gaussian samples to new data samples.}
\label{plot:2d_CNFs_vs_SchB}
\end{figure}

\subsection{2D Double Well Experiment}\label{sec:SOC}
In this section, we demonstrate our algorithm by solving the optimal control problem of transitioning from one minimizer to another, a typical application in the field of chemical reactions.Considering the double well potential, we demonstrate how to incorporate a functional prior of any form and learn the distribution of paths connecting $\rho_0$ and $\rho_1$. Similar experiments have been researched in \cite{gao2023optimal,vargas2021solving} . To incorporate prior knowledge, we use a potential field with two local minima, and the boundary distributions $\rho_0$ and $\rho_1$ are modeled as Gaussian distributions centered at the wells. 

The problem is equivalent to find the optimal control $u$ such that the diffusion process $dX_t= \left[u(X_t,t)-\nabla U(X_t) \right]dt + \sqrt{2} dW$ bridges the distributions $\rho_0$ and $\rho_1$ and minimizes the $L^2$ cost:
\begin{equation}\label{problem:SOC_exp}
\begin{array}{ll}
    \min\limits_{\rho, u} & \frac{1}{2}\int_{0}^{T} \int_{\Omega}\rho(x,t) |u(x,t)|^2 dx dt, \\
\text { s.t. } &\partial_{t} \rho +\nabla \cdot(\rho (u-\nabla U))= \sigma^2 \Delta \rho ,\\
&\rho(x,0)=\rho_{0}(x) ,\ \rho(x,1)=\rho_{1}(x) \\
\end{array}
\end{equation}

The aim of this experiment is to show that, with this prior, the SBP results in low-energy trajectories (as determined by the well's potential function) for particle configurations sampled from the wells. Intuitively, the learned paths are expected to avoid the high-energy peak at $x = (0, 0)$ and instead follow the "passes" on either side. Notably, if we were to estimate the optimal transport geodesics between $\rho_0$ and $\rho_1$ or use a Brownian motion prior, the learned trajectories would pass through the center, which corresponds to the highest energy path between the wells.

The algorithm remains nearly identical. We parameterized $f_{\theta}:= u-\sigma^2 \nabla \log \rho$ as a neural network and denotes $F(x,t):=f_{\theta}(x,t)-\nabla U(x)$. Consequently, the corresponding ODE system becomes:
\begin{equation}\label{eq:2D_doublewell_ODE}
\partial_{t}\left[\begin{array}{c}
z \\
\ell \\
s
\end{array}\right]=\left[\begin{array}{c}
F \\
\nabla\cdot F\\
-\nabla (\nabla \cdot F )- \mathcal{J}_{F}^{T} s
\end{array}\right], \quad\left[\begin{array}{c}
z(x, 0) \\
\ell(x,0)\\
s(x, T)
\end{array}\right]=\left[\begin{array}{c}
x \\
0\\
\nabla \log \rho_1(z(x,T))
\end{array}\right].
\end{equation}
The loss function remains unchanged. The parameterized function $f_{\theta}$ can be trained by minimizing this loss function and the ODE trajectories can be recovered by solving
\begin{equation*}
    dx_t=(f_{\theta}(x_t,t)-\nabla U(x_t))dt.
\end{equation*}
Additionally the optimal control can be recovered using score matching as described in Section \ref{subsection:recovering_drift}, enabling the generation of SDE paths for particles.

\begin{algorithm}
\caption{Solving 2D double well problems}
\begin{algorithmic}[1]
\Require Data samples $\{ x_{i}\}_{i=1}^{N}$, time interval $[0,T]$, initializing network $f_{\theta}$
\For{number of training iterations}
    \State Solving ODE system \eqref{eq:2D_doublewell_ODE} with $f_{\theta}$ to obtain $z(x_i,T)$, $\ell(x_i,T)$ and $s(x_i,t)$.
    \State Calculate cost function $\mathcal{J}=\alpha \mathcal{J}_{\mathbb{KL}}+ \mathcal{J}_{B} $
    \State Use ADAM optimizer optimizer to update network parameter of $f_{\theta}$
\EndFor
\State Use Algorithm \ref{alg:Drift_recovering} to recover the optimal control $u$ from $f_{\theta}$.
\end{algorithmic}
\end{algorithm}

Especially the potential and the marginal distributions in our experiment are set to be
\begin{equation}
    \begin{aligned}
        &U(x,y) = 0.3\left(\frac{1}{2}(x^2-1)^2+y^2 + 8\exp\left(\frac{-((x-a)^2+y^2)}{32}\right)\right) \\
        &\rho_0 = \mathcal{N}\left(\begin{bmatrix}-1\\0\end{bmatrix}, \begin{bmatrix}
            0.0125&0\\0&0.15
            \end{bmatrix}\right), \quad  \rho_1 = \mathcal{N}\left(\begin{bmatrix}
                1\\0
            \end{bmatrix}\begin{bmatrix}
                0.0125&0\\0&0.15
            \end{bmatrix}\right),
    \end{aligned}
\end{equation}
where $a$ is a constant controlling the potential barrier in the middle. The learned control vector field for the SDE at different time is shown in Figure~\ref{fig:velocity_fields}.The white lines indicate the contour lines (level curves) of the potential function, the black arrows denote the directions of the control vectors, and the background color represents the magnitude of the control vectors. We can observe that at $t=0$ the control vector in the left region is opposite to the gradient of the given potential, driving particles away from the left local minima. At $t=1$ the control vector in the right region aligns with the gradient of the potential, guiding particles toward the target local minima. Figure~\ref{fig:double-well} shows the ODE and SDE paths for different $a$. For $a=0.5$ we can observe that the path of particles bypassing the barrier. The case where $a=0$ presents an intriguing phenomenon: despite the presence of a symmetric potential field, the trained network consistently produces asymmetric results. Note that, however, the probability that the trained velocity field chooses the upper half region or lower half region is symmetric.This behavior suggests the presence of implicit regularization mechanisms within the neural network architecture. It is unclear at this point whether the original system has the physical symmetry breaking. Figure~\ref{fig:double-well-terminal} shows the terminal distribution of preset data and sampling from trained network. The terminal distributions matche the target distribution in both $a=0.5$ and $a=0$ cases, demonstrating that our algorithm can always find connecting paths that avoid high-energy peak and enable transitions between the minimizers

\begin{figure}
\centering
\begin{minipage}{\textwidth}
\centering
\includegraphics[width=\textwidth]{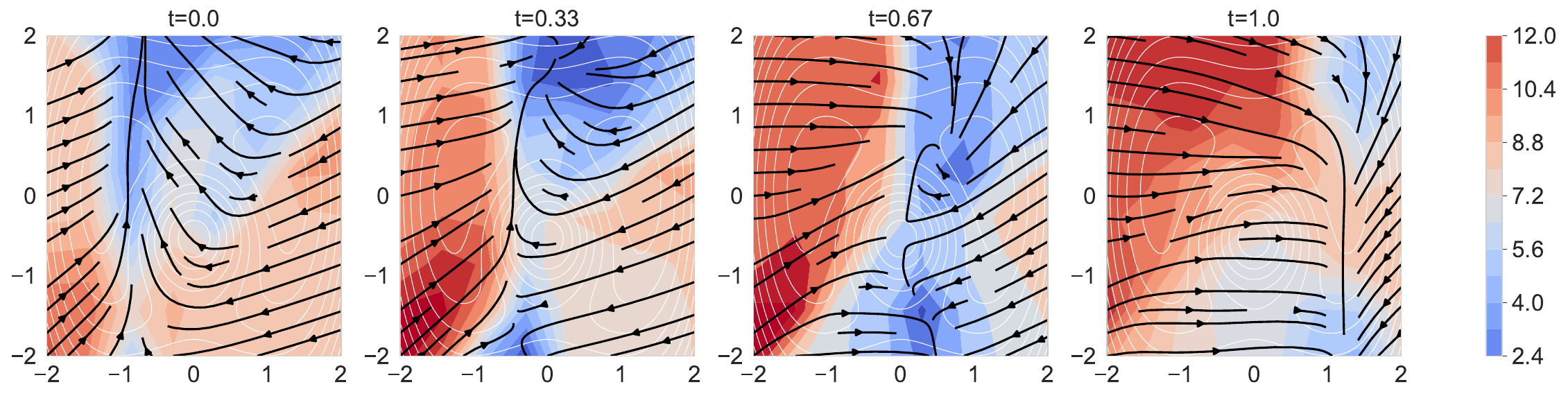}
\end{minipage}
\vspace{0.5cm}  
\begin{minipage}{\textwidth}
\centering
\includegraphics[width=\textwidth]{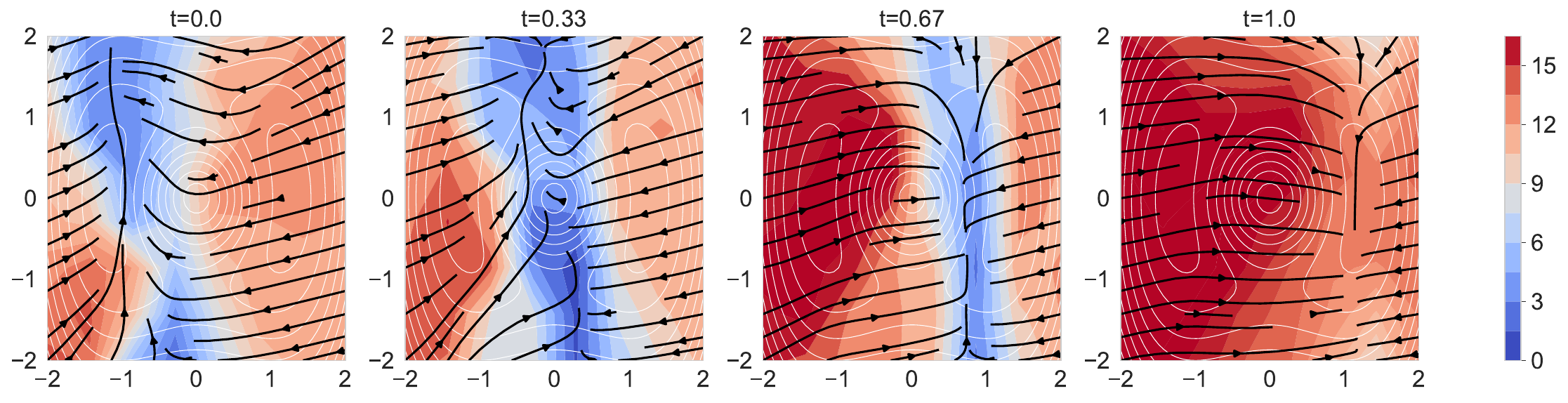}
\end{minipage}
\caption{Optimal control $u$ for different values of $a$. The white lines represent the contour lines of the potential function, the black arrows represent the directions of the control vectors, and the background color represents the magnitude of the control vectors. Upper: $a=0.5$. Bottom: $a=0$}
\label{fig:velocity_fields}
\end{figure}

\begin{figure}[htbp]
\centering
\subfigure[ODE trajectories for $a=0.5$]{%
\begin{minipage}[t]{0.35\linewidth}
\centering
\includegraphics[width=\linewidth]{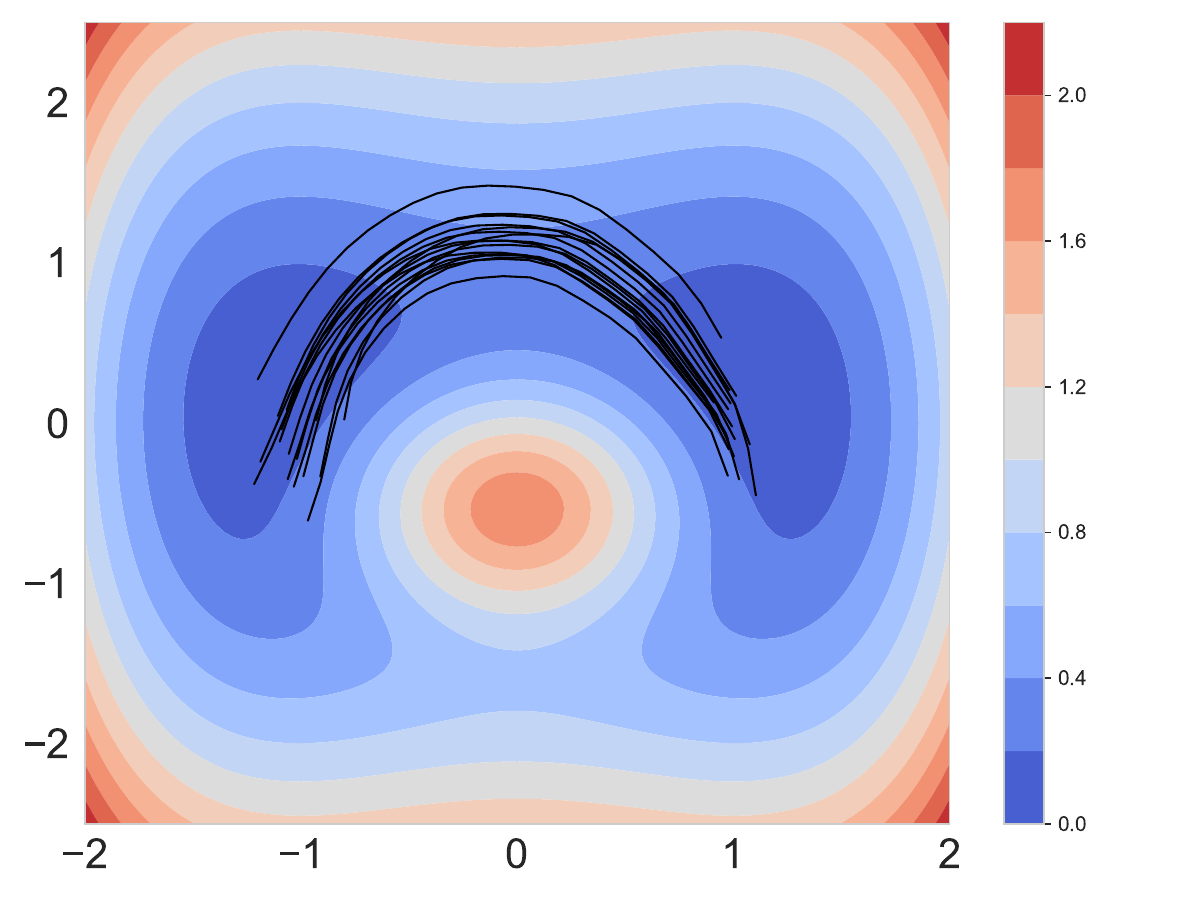}
\end{minipage}}%
\hspace{0.01\linewidth}
\subfigure[SDE paths for $a=0.5$]{%
\begin{minipage}[t]{0.35\linewidth}
\centering
\includegraphics[width=\linewidth]{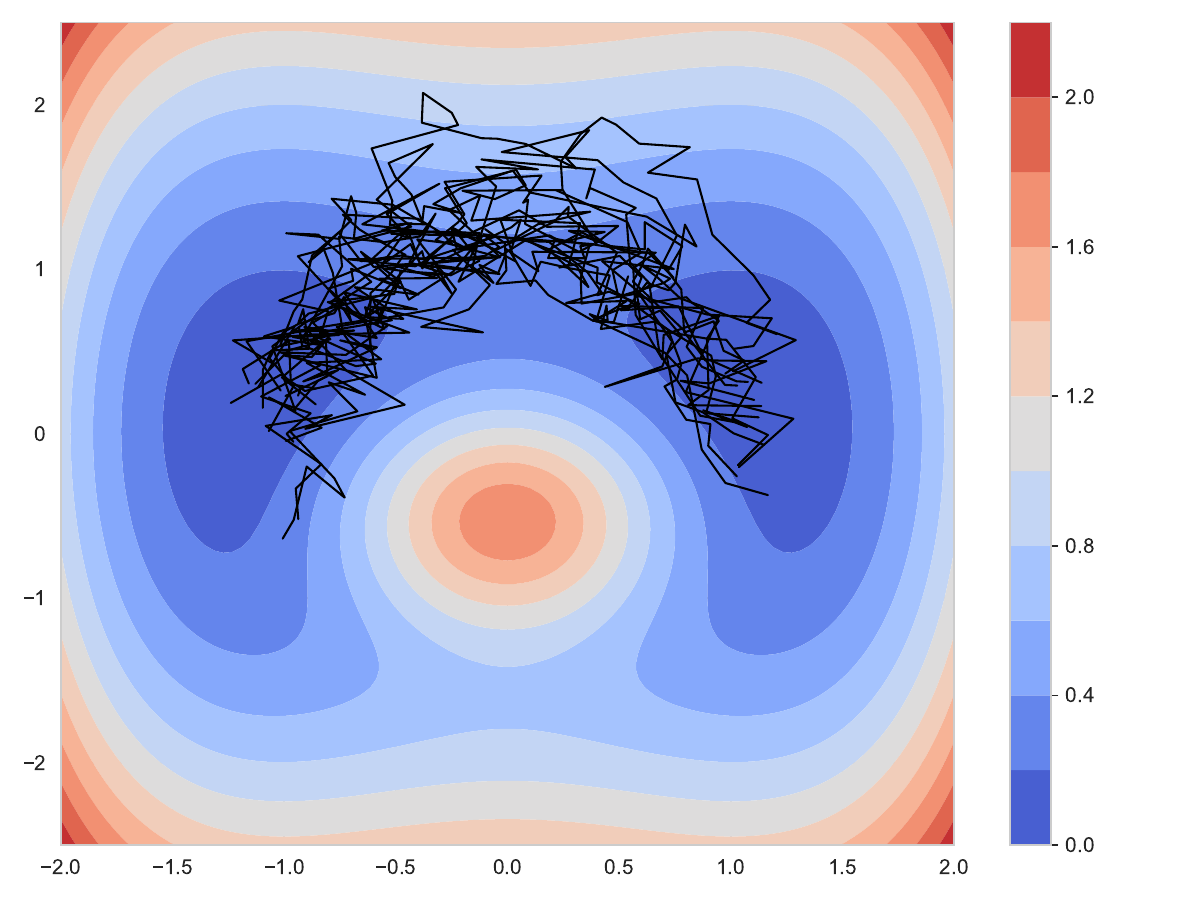}
\end{minipage}}

\subfigure[ODE trajectories for $a=0$]{%
\begin{minipage}[t]{0.35\linewidth}
\centering
\includegraphics[width=\linewidth]{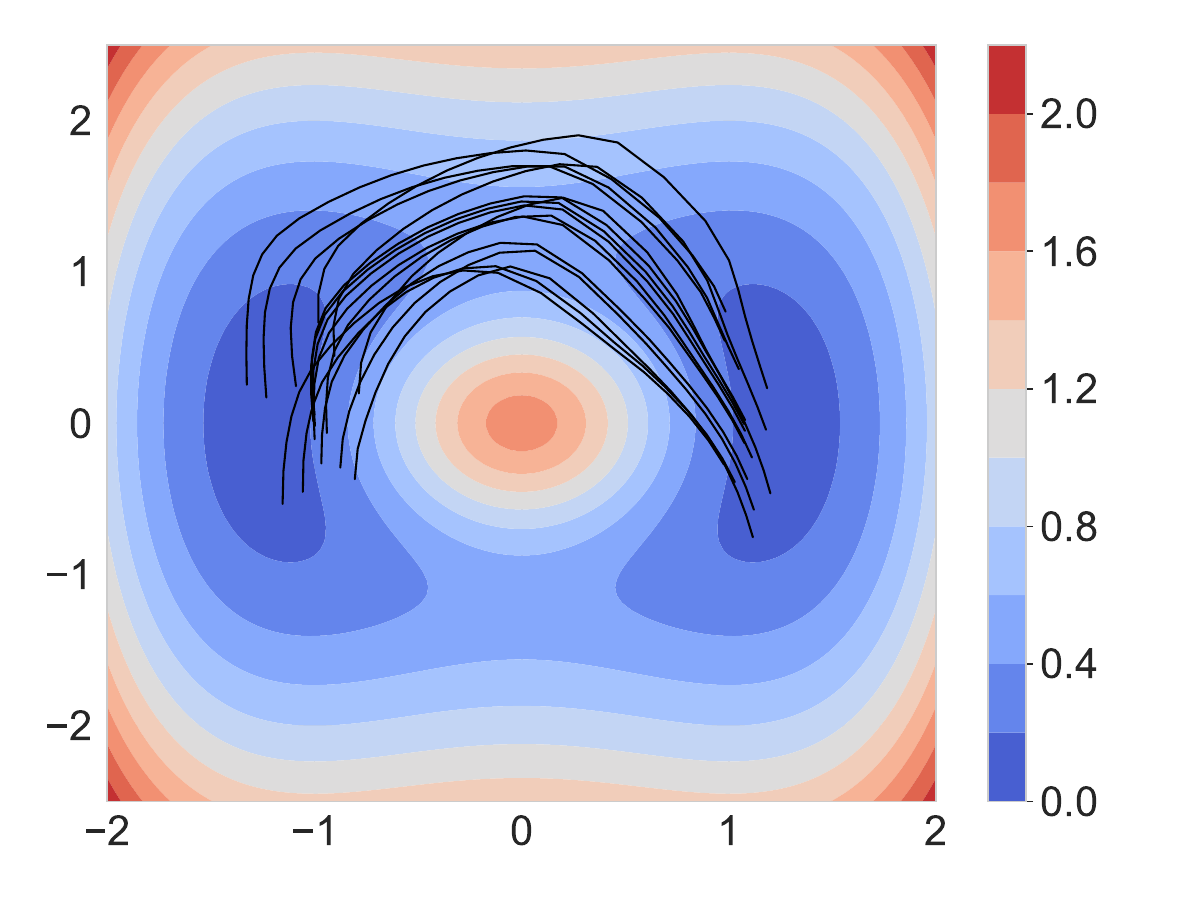}
\end{minipage}}%
\hspace{0.01\linewidth}
\subfigure[SDE paths for $a=0$]{%
\begin{minipage}[t]{0.35\linewidth}
\centering
\includegraphics[width=\linewidth]{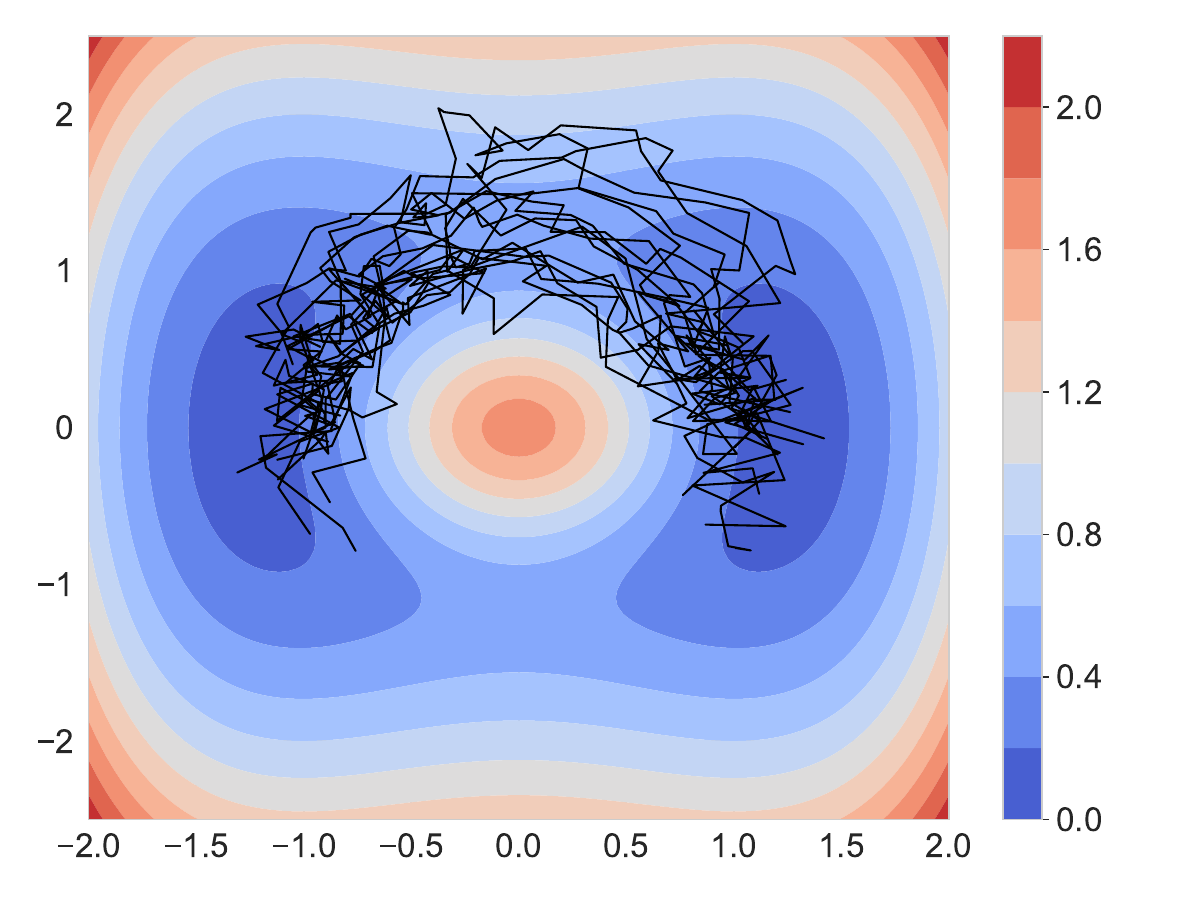}
\end{minipage}}

\caption{ODE trajectories and SDE interpolation paths for standard and symmetric breaking cases.}
\label{fig:double-well}
\end{figure}

\begin{figure}[htbp]
\centering
\subfigure[Target distribution $\rho_1$]{
\begin{minipage}[t]{0.3\linewidth}
\centering
\includegraphics[width=\textwidth]{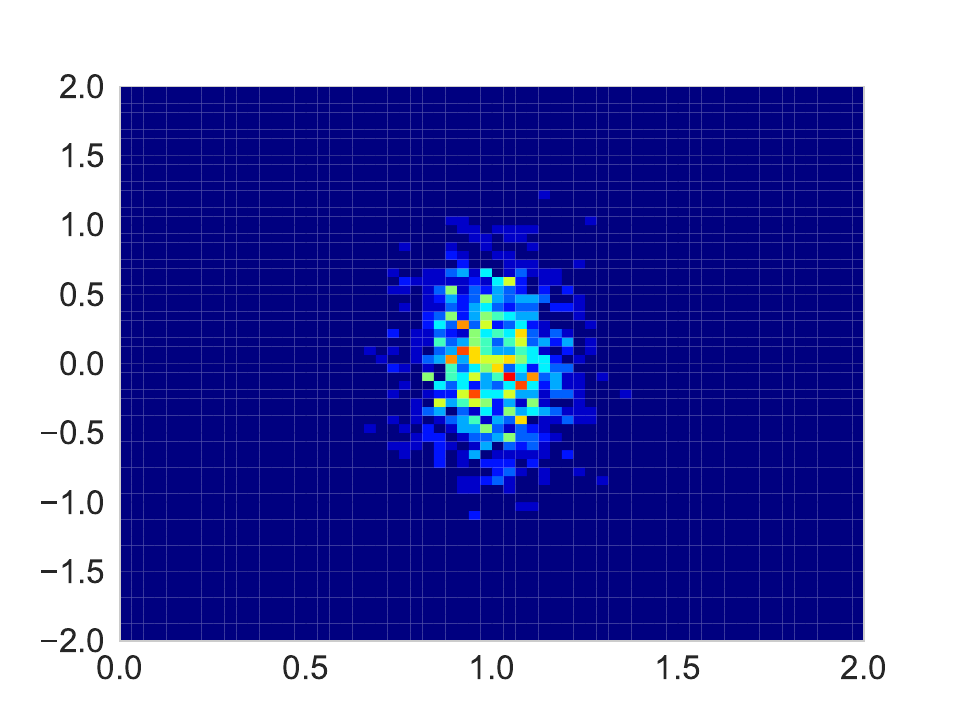}
\end{minipage}
}
\subfigure[Sampling for $a=0.5$]{
\begin{minipage}[t]{0.3\linewidth}
\centering
\includegraphics[width=\textwidth]{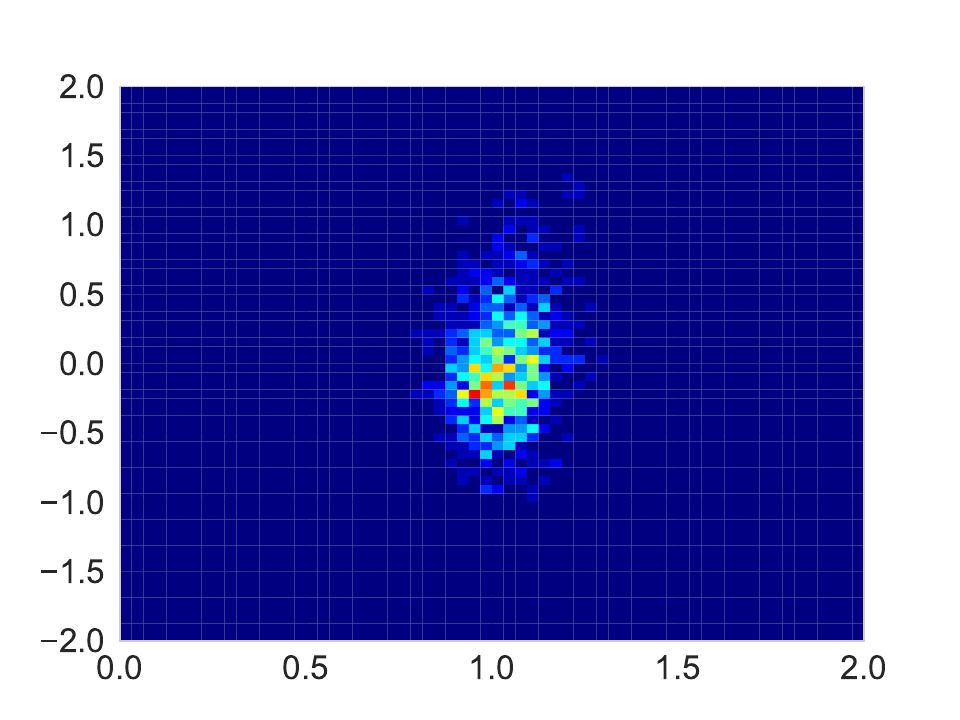}
\end{minipage}
}
\subfigure[Sampling for $a=0$]{
\begin{minipage}[t]{0.3\linewidth}
\centering
\includegraphics[width=\textwidth]{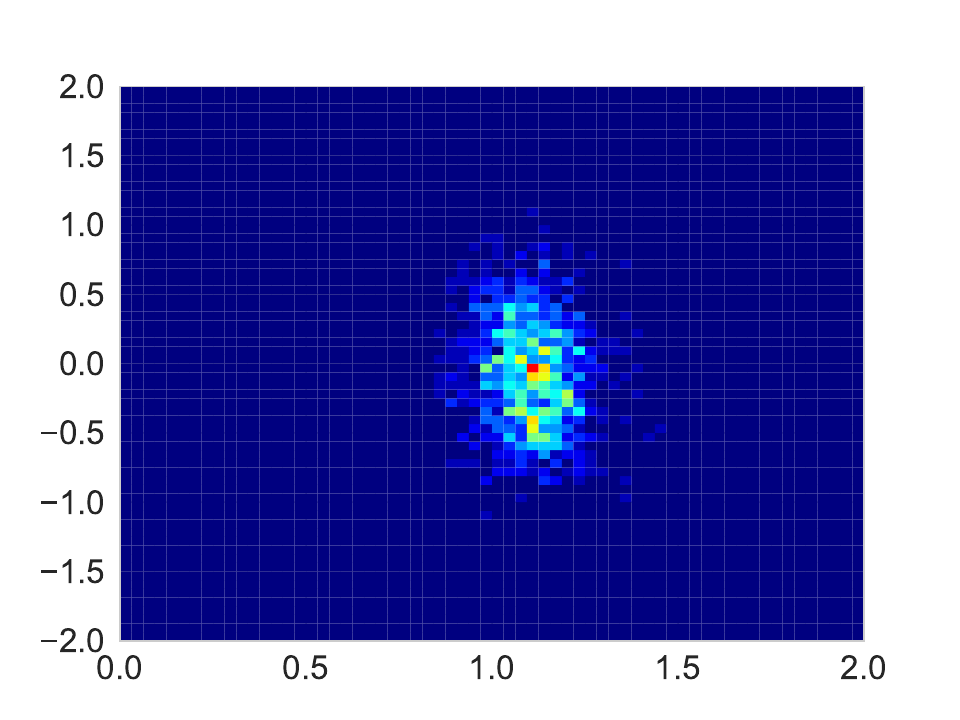}
\end{minipage}
}
\caption{Comparison of terminal distribution.}
\label{fig:double-well-terminal}
\end{figure}

\section{Conclusion and discussion}\label{sec:dis}
In summary, we have developed a framework based on CNFs and score matching to solve the dynamic formulation of the SBP in the first part. Specifically, we introduce a hypothetical velocity field to transform the Fokker-Planck equation into a transport equation, enabling the use of the CNF framework to solve the dynamic SBP and track density evolution. In the second part, we employ $\Gamma$-convergence to establish the convergence of minimizers from our neural network solutions to the theoretical ones as the regularization coefficient $\alpha \to \infty $. Our approach provides a neural network-based solution to the SBP and contributes to a deeper understanding of the convergence properties of CNF models with regularization. The convergence analysis also offers theoretical guarantees for the stability during the training of our algorithm. Future research may explore the application of analogous methodologies to design algorithms and establish convergence analyses for addressing SDE-type mathematical tasks, such as solving optimal control problem and tracking evolution of physical systems.

\subsection*{Acknowledgement}
This work is partially supported by the National Key R\&D Program of China No. 2020YFA0712000 and No. 2021YFA1002800. The work of L. Li was partially supported by NSFC 12371400 and 12031013, Shanghai Municipal Science and Technology Major Project 2021SHZDZX0102,  and Shanghai Science and Technology Commission (Grant No. 21JC1403700, 21JC1402900).

\section*{Appendix}
\appendix

\section{Missing proofs}\label{app:severalproofs}

\begin{proof}[Proof of Proposition \ref{pro:gammaconv}]
According to Proposition \ref{prop:connection}, it suffices to show that
\begin{equation}
\inf_{x_n \to x} \limsup_{n\rightarrow \infty} f_n(x_n) \leq f(x)\leq \inf_{x_n \rightarrow x} \liminf_{n\rightarrow \infty} f_n(x_n).
\end{equation}

Consider $x_n=x$, $f(x)=\lim\limits_{n\to\infty}f_n(x)=\limsup\limits_{n\to\infty}f_n(x_n)\geq \inf\limits_{x_n\to x}\limsup\limits_{n\to\infty}f_n(x_n)$.

For the other inequality, consider any $x_n\to x$.
Since $f_n\uparrow f$ pointwise,
\begin{equation}
f(x)=\lim_{n\to\infty}f_n(x)\leq\lim_{n\to\infty}\liminf_{m\to\infty}f_n(x_m)\leq\liminf_{m\to\infty}f_m(x_m).
\end{equation}
The last inequality holds since for any fixed $n$ and any $m$ larger than $n$, one has
$f_n(x_m)\le f_m(x_m)$. Hence $f_n$ is $\Gamma$-convergent to $f$.
\end{proof}

\begin{proof}[Proof of Proposition \ref{proposition:score_ODE}]
For the second equation,
\begin{equation*}
\begin{split}
\frac{\partial l(x, t)}{\partial t}  &= \frac{1}{\det (\nabla z(x,t))} \frac{\partial \det (\nabla z(x,t))}{\partial t} \\
&= \frac{1}{\det (\nabla z(x,t))} \cdot \det (\nabla z(x,t)) \cdot \mathrm{tr} \left[ (\nabla z(x,t))^{-1} \frac{\partial \nabla z(x,t)}{\partial t}  \right]\\
&=\frac{1}{\det (\nabla z(x,t))} \cdot \det (\nabla z(x,t)) \cdot \mathrm{tr} \left[ (\nabla z(x,t))^{-1} \nabla z(x,t)  \nabla_z f(z(x,t),t)    \right] \\
&=\mathrm{tr} \left[ ( \nabla_z f(z(x,t),t) \right]\\
&= \nabla_z \cdot f(z(x,t),t), 
\end{split}
\end{equation*}
where we have used following identities
\begin{equation*}
     \frac{\partial \det (A)}{\partial t}=\det (A) \cdot \mathrm{tr} \left[  A^{-1}  \frac{\partial A}{\partial t}\right], \quad \mathrm{tr} (AB)= \mathrm{tr} (BA).
\end{equation*}
For the third equation, we first have
\begin{equation}\label{appendix:eq1}
\begin{aligned}
\frac{\mathrm{d}}{\mathrm{~d} t} \nabla_z \log \rho(z(x,t),t)  & =\frac{\partial}{\partial t} \nabla_z \log \rho(z,t) |_{z=z(x,t)}+\frac{\mathrm{d} z}{\mathrm{d} t} \frac{\partial}{\partial z} \nabla_z \log \rho(z,t) |_{z=z(x,t)} \\
& =\nabla_z \frac{\partial}{\partial t} \log \rho(z,t)|_{z=z(x,t)}+ f(z,t) \nabla_z^2 \log \rho(z,t)|_{z=z(x,t)}.
\end{aligned}
\end{equation}
Using the Fokker-Plank equation \eqref{Continuituy_equation_SchB} and $f=u- \sigma^2 \nabla \log \rho$, we derive
\begin{equation}
\partial_t \log \rho(z,t)=-\nabla_z \cdot f(z,t)-f(z,t) \cdot \nabla_z \log \rho(z,t),
\end{equation}
which together with
\begin{equation}
\nabla_z\left( \nabla_z \log \rho(z,t) \cdot f(z,t)  \right)=\left[ \nabla_z^2 \log \rho(z,t) \right]\cdot f(z,t) +\nabla_z f(z,t)\cdot \nabla_z \log \rho(z,t),
\end{equation}
taking it into \eqref{appendix:eq1} one get
\begin{equation}
\frac{\mathrm{d}}{\mathrm{~d} t} \nabla_z \log \rho(z(x,t),t) =-\nabla_z ( \nabla_z \cdot f(z,t)) |_{z=z(x,t)}-\nabla_z f(z,t) \cdot\nabla \log \rho(z,t)  |_{z=z(x,t)}.
\end{equation}
\end{proof}

\section{Implementation details for solving  SBP with optimization method}\label{Appendix:primal_dual}
To validate our algorithm, we use a primal-dual hybrid algorithm to solve the SBP on the grid.This algorithm has been previously employed to compute the Earth Mover's Distance \cite{li2018parallel} and the spherical Wasserstein-Fisher-Rao metric \cite{jing2024machine}, as well as to verify the convergence from optimal transport to unbalanced optimal transport \cite{xiong2023convergence}. We begin by taking $\sigma =1$ for simplicity and reformulating the optimization problem into a convex form:
\begin{equation}
\min _{\rho, m}\left\{ \int_0^T \int_{\mathbb{R}^d} \frac{|m|^2}{2 \rho} d x d t, \partial_t \rho+\nabla \cdot m =\Delta \rho , \rho(x,0)=\rho_0(x), \rho(x,T)=\rho_1(x) \right\}.
\end{equation}
We then focus on solving the corresponding min-max problem:
\begin{equation}\label{eq:min-max}
\max_{\phi}\min_{\substack{\rho,m \\  \rho(x,0)=\rho_0(x)\\ \rho(x,T)=\rho_1(x)}} \int_0^T \int_{\Omega} \frac{|m|^2}{2 \rho}d x d t+ \int_0^T\int_{\Omega} \phi(x,t)(\partial_t \rho+\nabla \cdot m-\Delta \rho) dx dt .
\end{equation}
We employ the same discrete scheme for divergence operator and boundary conditions as \cite{xiong2023convergence}. For simplicity, we consider the space domain $\Omega=[0,1]^{d}$ and the time domain $[0,1]$. Let $\Omega_{h}$ be the discrete space mesh-grid of $\Omega$ with step size $h$, i.e. $\Omega_{h}=\{0,h,2h,\cdots,1\}^d$. The time domain $[0,1]$ is discreted with step size $\Delta t$ . Let $N_x=1/h$ be the space grid size and $N_t=1/\Delta t$ be the time grid size. All optimization variables ($\rho$, $m$ and $\phi$ ) will be defined on the grid.To ensure clarity, we first provide some definitions on the discrete space $\Omega_h$:
\begin{equation*}
    \begin{aligned}
        \int_{\Omega_h}f(x) &:=\sum_{x \in \Omega_{h},x_i \neq 0} f(x) h^d, \quad \left\langle f, g \right\rangle_h :=  \int_{\Omega_h}f(x) g(x) , \\
        \nabla_{h} \cdot m(x) &:=\sum_{i=1}^{d} D_{h,i} m(x) , x\in \Omega_h, \\
        \Delta_h \rho &:=\frac{1}{h^2} \sum_{i=1}^{d} \rho \left(x_1,\cdots,  x_i+h, \cdots , x_d\right)-2\rho \left(x_1,\cdots,  x_i, \cdots , x_d\right)+\rho \left(x_1,\cdots,  x_i-h, \cdots , x_d\right),
    \end{aligned}
\end{equation*}
where $D_{h,i}$ denotes discrete differential operator for $i$-th component in $i$-th dimension with step size $h$:
\begin{equation*}
D_{h, i} m(x)=\left(m_i \left(x_1,\cdots,  x_i+h, \cdots , x_d\right)-m_i \left(x_1,\cdots,  x_i, \cdots , x_d\right)\right) / h, \quad 0 \leq x_i \leq 1-h. \\
\end{equation*}

From the discretized form, we can describe the sizes of the optimization variables individually. The sizes of discretized variables are :$(N_t +1)\times (N_x)^d $ for $\rho$, $ N_t \times (N_x+1)^{d} \times d $ for $m$ and $ N_t \times  (N_x)^{d} $ for $\phi$. The boundary conditions are given as $\rho_1=\mu$, $\rho_{N_t+1}=\nu$ and $\left( m-\nabla \rho \right) \cdot \vec{n}=0$ for all $x \in \partial \Omega_h$. Then the discretization of \eqref{eq:min-max} can be written as:
\begin{equation}\label{eq:discrete_minmax}
\max_{\phi}\min_{\rho,m,\xi}L\left(m, \rho,\phi \right)= \sum_{t=1}^{N_t} \left\{ \int_{\Omega_h} \frac{|m_{t}|^2}{2\rho_{t}} d x +  \left\langle \phi_{t},\frac{\rho_{t+1}-\rho_{t}}{\Delta t}+\nabla_h \cdot m_{t}-\Delta_h \rho_{t} \right\rangle_h  \right\},
\end{equation}
One can add ghost points at the boundary and use discretized boundary condition to deal with central difference discretion for $\rho$, which ensures that the problem is well-defined. For instance, consider the left boundary: one can add a ghost point $\rho_t^{-1}$ and the boundary condition yields
\begin{equation*}
    m_t^{0}-\frac{\rho_{t}^{0}-\rho_{t}^{-1}}{h}=0,
\end{equation*}
which can then be incorporated into the Laplacian operator:
\begin{equation*}
    \Delta_h \rho_t^{0}=\frac{\rho_t^{-1}-2\rho_t^{0}+\rho_t^{1}}{h^2}=\frac{-\rho_t^{0}+\rho_t^{1}}{h^2}-\frac{m_t^0}{h}.
\end{equation*}
As a result, \eqref{eq:discrete_minmax} is independent of $\rho_t^{-1}$.

Then, the primal-dual hybrid algorithm proceeds with the following three steps for the variables with subscript $t$:
\begin{itemize}
    \item $\left(m_{t}^{k+1},\rho_t^{k+1}\right)=\underset{m^{\star},\rho^{\star}}{\arg \min } L\left(m^{\star},\rho^{\star},\phi^{k} \right)+\frac{1}{2 \mu}\left(\left\|m^{\star}-m^{k}_{t}\right\|_{2}^2+\left\|\rho^{\star}-\rho^k_t\right\|_{ 2}^2  \right),$
    \item  $\tilde{m}_t^{k+1}=2 m_t^{k+1}-m_t^k, \quad \tilde{\rho}_t^{k+1}=2 \rho_t^{k+1}-\rho_t^k,$
    \item $\phi_t^{k+1}=\underset{\phi^{\star}}{\arg \max } L\left(\tilde{m}_t^{k+1}, \tilde{\rho}_h^{k+1},\phi^{\star}\right)-\frac{1}{2 \tau} \left( \left\|\phi^{\star}-\phi^{k}_t\right\|_{2}^2   \right).$
\end{itemize}
The first step of the algorithm is to update $\rho$ and $m$, which is equivalent to solve the following equations:
\begin{equation}
\label{PD_algorithm_step1}
\left\{\begin{array}{l}
m^{\star}=\frac{\rho^{\star}(m_t^{k}-\mu \text{div}_{h}^{\ast} \phi_t^{k})}{\rho^{\star}+\mu}, \\
-\frac{{m^{\star}}^{2}}{2{\rho^{\star}}^2}+\frac{\phi_{t-1}^{k}-\phi^{k}_{t}}{\Delta t}- \Delta_h^{\star} \phi^{k}_{t}+\frac{1}{\mu}(\rho^{\star}-\rho_t^k)=0,
\end{array}\right.
\end{equation}
where $\text{div}_{h}^{\ast}$ represents the conjugate operator of divergence operator and $\Delta_h^{\star}$ represents the conjugate operator of Laplacian operator. By the definition of conjugate operator, we have
\begin{equation*}
    \begin{aligned}
        &\left\langle \nabla_h \cdot f ,g \right\rangle_{h} =\left\langle f ,\text{div}_{h}^{\ast} g \right\rangle_h= \left\langle f , -\nabla_h g \right\rangle_h,\\
        &\left\langle \Delta_h \rho ,\phi \right\rangle_{h}=\left\langle  \rho , \Delta_h^{\star}\phi \right\rangle_{h}=\left\langle  \rho ,\Delta_h  \phi \right\rangle_{h}.
    \end{aligned}
\end{equation*}
Thus $\text{div}_{h}^{\ast}=-\nabla_{h}$ and $\Delta_h^{\ast}= \Delta_h$, where $\nabla_h u =\left(\partial_{h, 1} u, \partial_{h, 2} u, \cdots, \partial_{h, d} u\right)$ and each $\partial_{h,i} $ denotes discrete differential operator in $i$-th dimension with step size $h$:
\begin{equation*}
\partial_{h, i} u(x)= \left(u\left(x_1,\cdots,  x_i+h, \cdots , x_d\right)-u\left(x_1,\cdots,  x_i, \cdots , x_d\right)\right) / h, \quad 0\leq x_i \leq 1-h .
\end{equation*}

Meanwhile, the term $\phi_{t-1}^{k}$ in \eqref{PD_algorithm_step1} is undefined when $t = 1$. Observe that the term $\frac{\phi_{t-1}^{k} - \phi^{k}_{t}}{\Delta t}$, arising from taking the variation of \eqref{eq:discrete_minmax} with respect to $\rho_t$, which should be $-\frac{\phi^{k}_{1}}{\Delta t}$ in the case of $t=1$. This naturally suggests the introduction of a boundary condition $\phi_0 = 0$.  Solving above system \eqref{PD_algorithm_step1} requires us to solve the roots for a third order polynomial, where $\rho^{\star}$ should be the largest real root. The third step of the algorithm is to update dual variables:
\begin{equation}
\phi_t^{k+1}=\phi_{t}^{k}+ \tau \left( \frac{\tilde{\rho}_{t+1}^{k+1}-\tilde{\rho}_{t}^{k+1}}{\Delta t}+\nabla_h \cdot \tilde{m}_{t}^{k+1}- \Delta_h \tilde{\rho}_t^{k+1}\right).
\end{equation}

\normalem
\bibliographystyle{plain}
\bibliography{SchB}

\end{document}